\documentclass[a4paper,12pt]{amsart}
\usepackage[lmargin=3.7cm,rmargin=2.7cm,bmargin=3.5cm,tmargin=3cm]{geometry}

\usepackage[latin1]{inputenc}
\usepackage[english]{babel}
\usepackage{graphicx}

\usepackage{amsmath}
\usepackage{amssymb}  
\usepackage{amsthm}
\usepackage{bbm}
\usepackage{psfrag}

\usepackage{enumerate}

\usepackage{textcomp}

\usepackage[numbers,sort]{natbib}

\allowdisplaybreaks[3]

\renewcommand{\P}{\mathbb{P}}
\newcommand{\charfun}[1]{\mathbbm{1}_{#1}}
\newcommand{\set}[1]{{#1}}
\newcommand{\Pcond}[2]{\mathbb{P}\left(\left.#1\;\right|\;#2\right)}
\newcommand{\coll}[1]{\mathcal{#1}}
\newcommand{\uarg}{\,\cdot\,}
\newcommand{\ud}{\mathrm{d}}
\newcommand{\defeq}{\mathrel{\mathop:}=} 
\newcommand{\R}{\mathbb{R}}
\newcommand{\Econd}[2]{\mathbb{E}\left[\left.#1\;\right|\;#2\right]}
\newcommand{\Eop}[1]{\mathbb{E}\left[#1\right]}
\newcommand{\E}{\mathbb{E}}

\newcommand{\fixprop}{q_{\text{fix}}}

\newcommand{\Cov}{\mathop{\mathrm{Cov}}}
\newcommand{\F}{\coll{F}}
\newcommand{\G}{\coll{G}}

\newtheorem{theorem}{Theorem}
\newtheorem{corollary}[theorem]{Corollary}
\newtheorem{lemma}[theorem]{Lemma}

\theoremstyle{definition}

\newtheorem{assumption}[theorem]{Assumption}

\theoremstyle{remark}
\newtheorem{remark}[theorem]{Remark}

\date{\today}
\begin{document}

\title[Can the Adaptive Metropolis Collapse]%
{Can the Adaptive Metropolis Algorithm Collapse Without the Covariance Lower
  Bound?}
\author{Matti Vihola}
\address{Matti Vihola, Department of Mathematics and Statistics,
  University of Jyväskylä,
  P.O.Box 35 (MaD),
  FI-40014 University of Jyväskylä,
  Finland}
\email{matti.vihola@iki.fi}
\urladdr{http://iki.fi/mvihola/}
\thanks{The author was supported 
  by the Academy of Finland, projects no.~110599 and 201392,
  by the Finnish Academy of Science and Letters, Vilho,
  Yrjö and Kalle Väisälä Foundation,
  by the Finnish Centre of Excellence in Analysis and
  Dynamics Research,
  and by the Finnish Graduate School in
  Stochastics and Statistics.}


\begin{abstract} 
   The Adaptive Metropolis (AM) algorithm is based on
   the symmetric random-walk Metropolis algorithm. The proposal 
   distribution has the following time-dependent
   covariance matrix at step $n+1$
   \[
       S_n = \Cov(X_1,\ldots,X_n) + \epsilon I,
   \]
   that is, the sample covariance matrix of the history of the
   chain plus a (small) constant $\epsilon>0$ multiple of the identity matrix
   $I$. The lower bound on the eigenvalues of $S_n$
   induced by the factor $\epsilon I$ is theoretically convenient,
   but practically cumbersome, as a good value for the parameter $\epsilon$
   may not always be easy to choose.  This article considers variants of the AM
   algorithm that do not explicitly bound the eigenvalues of $S_n$ away
   from zero. 
   The behaviour of $S_n$ is studied in detail,
   indicating that the eigenvalues of $S_n$ do not tend 
   to collapse to zero in general. 
   In dimension one, it is shown that $S_n$ is bounded away from zero 
   if the logarithmic target density is uniformly continuous.
   For a modification of the AM algorithm including an additional 
   fixed component in the
   proposal distribution, the eigenvalues of $S_n$ are shown to stay away
   from zero with a practically non-restrictive condition. This result implies a strong 
   law of large numbers for super-exponentially decaying target distributions with regular
   contours.
\end{abstract} 

\maketitle

\section{Introduction} 
\label{sec:intro} 

Adaptive Markov chain Monte Carlo (MCMC) methods have attracted increasing
interest in the last few years, after the original work of Haario, Saksman,
and Tamminen \cite{saksman-am}
and the subsequent advances in the field
\cite{andrieu-robert,atchade-rosenthal,andrieu-moulines,roberts-rosenthal};
see also the recent review \cite{andrieu-thoms}. Several adaptive MCMC algorithms 
have been proposed up to date, but the 
seminal Adaptive Metropolis (AM) algorithm \cite{saksman-am} is still one of
the most applied methods, perhaps due to its simplicity and
generality. 

The AM algorithm is a symmetric random-walk Metropolis algorithm, with
an adaptive proposal distribution. The algorithm starts\footnote{
  The initial `burn-in' phase included in the original algorithm is not
  considered here.}
at some point $X_1\equiv
x_1\in\R^d$ with an initial positive definite covariance matrix
$S_1\equiv s_1\in\R^{d\times d}$
and follows the recursion
\begin{enumerate}[(S1)]
    \item 
      \label{item:proposal}
      Let $Y_{n+1} = X_n + \theta S_n^{1/2} W_{n+1}$, where $W_{n+1}$ is an
      independent standard Gaussian random vector
      and $\theta>0$ is a constant.
    \item 
      \label{item:accept-reject}
      Accept $Y_{n+1}$ with probability
      $\min\big\{1,\frac{\pi(Y_{n+1})}{\pi(X_n)}\big\}$ and let $X_{n+1} = Y_{n+1}$;
      otherwise reject $Y_{n+1}$ and let $X_{n+1} = X_n$.
    \item 
      \label{item:adapt}
      Set $S_{n+1} = \Gamma(X_1,\ldots,X_{n+1})$.
\end{enumerate}
In the original work \cite{saksman-am}
the covariance parameter is computed by
\begin{equation}
    \Gamma(X_1,\ldots,X_{n+1}) =
    \frac{1}{n} \sum_{k=1}^{n+1} (X_k - \overline{X}_{n+1})
    (X_k - \overline{X}_{n+1})^T + \epsilon I,
    \label{eq:orig-am-crec}
\end{equation}
where $\overline{X}_n \defeq n^{-1}\sum_{k=1}^n X_k$ stands for the mean.
That is, $S_{n+1}$ is a covariance estimate of the
history of the `Metropolis chain' $X_1,\ldots,X_{n+1}$ plus a small
$\epsilon>0$ multiple of the identity matrix $I\in\R^{d\times d}$. 
The authors prove a strong law of
large numbers (SLLN) for the algorithm, that is, $n^{-1} \sum_{k=1}^n
f(X_k) \to \int_{\R^d} f(x) \pi(x) \ud x$ almost surely as $n\to\infty$ for
any bounded functional $f$
when the target distribution $\pi$ is bounded and compactly supported.
Recently, SLLN was shown to hold also for $\pi$ with unbounded support,
having super-exponentially decaying tails with regular contours and $f$
growing at most exponentially in the tails \cite{saksman-vihola}.

This article considers the original AM algorithm
(S\ref{item:proposal})--(S\ref{item:adapt}), without the lower bound
induced by the factor
$\epsilon I$. The proposal covariance 
function $\Gamma$, defined
precisely in Section \ref{sec:notations},
is a consistent covariance estimator first proposed in \cite{andrieu-robert}.
A special case of this estimator behaves asymptotically like the sample
covariance in \eqref{eq:orig-am-crec}.
Previous results indicate that
if this algorithm is modified by truncating the eigenvalues of $S_n$ 
within explicit lower and upper bounds,
the algorithm can be verified in a fairly general setting 
\cite{roberts-rosenthal,atchade-fort}. 
It is also possible to determine an increasing sequence of truncation sets
for $S_n$,
and modify the algorithm to include a re-projection scheme in order to verify the 
validity of the algorithm \cite{andrieu-moulines}.

While technically convenient, such pre-defined bounds 
on the adapted covariance matrix $S_n$ are inconvenient in practice.
Ill-defined values can affect the efficiency of the
adaptive scheme dramatically, rendering the algorithm useless in the
worst case. In particular,
if the factor $\epsilon>0$ in the AM algorithm is selected too
large, the smallest eigenvalue of the true covariance matrix of $\pi$ may
be well smaller than $\epsilon>0$, and the chain $X_n$ is likely to mix poorly.
Even though the re-projection scheme of \cite{andrieu-moulines} avoids
such behaviour by increasing truncation sets, which eventually contain
the desirable values of the adaptation parameter, the practical
efficiency of the algorithm is still strongly affected by the choice of
these sets \cite{andrieu-thoms}.

After defining precisely the algorithms in Section \ref{sec:notations},
the above mentioned unconstrained
AM algorithm is analysed in
Section \ref{sec:am}.
First, it is studied how
the AM algorithm run on an improper uniform target $\pi\equiv c>0$ behaves.
It is also shown that in a one-dimensional setting and with a uniformly continuous 
$\log\pi$, the variance parameter $S_n$ is bounded away from zero.
This fact is shown to imply, with the results in \cite{saksman-vihola}, 
a SLLN in the particular case of a Laplace target distribution.
While this result
has little practical value in its own right, it is the
first case where the unconstrained AM algorithm is shown to preserve
the correct ergodic properties. It shows that the algorithm possesses self-stabilising properties and further
strengthens the belief that the algorithm would be stable and ergodic under a
more general setting. The results of Section \ref{sec:am} 
also give some insight to the behaviour
of the adaptive chain  that can be helpful when the algorithm is applied
in practice.

Section \ref{sec:fcam} considers a
slightly different variant of the AM algorithm, due to Roberts and Rosenthal
\cite{roberts-rosenthal-examples}, replacing (S\ref{item:proposal}) 
with
\begin{enumerate}
    \item[(S1')]
      \label{item:proposal-mix}
      With probability $\beta$, let $Y_{n+1} = X_n + V_{n+1}$
      where $V_{n+1}$ is an independent sample of $\fixprop$; otherwise, let
      $Y_{n+1} = X_n + \theta S_n^{1/2} W_{n+1}$ as in (S\ref{item:proposal}).
\end{enumerate}
While omitting the parameter $\epsilon>0$,
the proposal strategy (S\ref{item:proposal-mix}') includes two additional
parameters: the mixing probability $\beta\in(0,1)$ and the fixed symmetric proposal
distribution $\fixprop$.
It has the advantage that 
the `worst case scenario' having ill-defined $\fixprop$ only `wastes' the
fixed proportion $\beta$ of samples, while $S_n$ can take any
positive definite value on adaptation.
This approach is analysed also in the 
recent preprint \cite{bai-roberts-rosenthal}, relying on a technical
assumption that ultimately implies that $X_n$ is bounded in probability. 
In particular, the authors show that if $\fixprop$ is a uniform density
on a ball having a large enough radius, then the algorithm is ergodic.
Section \ref{sec:fcam} uses a perhaps more transparent argument to
show that the proposal strategy
(S\ref{item:proposal-mix}') with a mild additional 
condition implies a sequence $S_n$ with eigenvalues bounded away from zero.
This fact implies a SLLN using the technique of \cite{saksman-vihola}, as
shown in the end of Section \ref{sec:fcam}. 


\section{The General Algorithm} 
\label{sec:notations} 

Let us define a Markov chain  $(X_n,M_n,S_n)_{n\ge 1}$ 
evolving in space $\R^d\times\R^d\times\mathcal{C}^d$ with the state space
$\R^d$ and
$\mathcal{C}^d \subset \R^{d\times d}$ 
standing for the positive definite matrices.
The chain starts at an initial position $X_1\equiv x_1\in\R^d$, with an initial
mean\footnote{A customary choice is to set $m_1 = x_1$.}
$M_1\equiv m_1\in\R^d$
and an initial covariance matrix
$S_1\equiv s_1\in\mathcal{C}^d$.
For $n\ge 1$, the chain is defined through the recursion
\begin{eqnarray}
    X_{n+1} & \sim &  P_{q_{S_{n}}}(X_n, \uarg) 
    \label{eq:x-rec} \\
    M_{n+1} & \defeq & (1-\eta_{n+1}) M_n + \eta_{n+1} X_{n+1} \label{eq:m-rec} \\
    S_{n+1} & \defeq & (1-\eta_{n+1}) S_n 
    + \eta_{n+1} (X_{n+1}-M_n)(X_{n+1}-M_n)^T.
    \label{eq:c-rec}
\end{eqnarray}
Denoting the natural filtration of the chain as
$\F_n \defeq \sigma(X_k,M_k,S_k:1\le k\le n)$, the notation in
\eqref{eq:x-rec} reads that $\Pcond{X_{n+1}\in
  A}{\F_n} = P_{q_{S_{n}}}(X_n, A)$ for any measurable $A\subset\R^d$.
The Metropolis transition kernel $P_q$ is defined
for any symmetric probability 
density $q(x,y) = q(x-y)$ through
\begin{multline*}
P_{q}(x,\set{A}) \defeq 
\charfun{\set{A}}(x) \left[ 1-
\int \min\left\{1,\frac{\pi(y)}{\pi(x)}\right\} q(y-x) \ud y\right]
\\ + \int_{\set{A}} 
  \min\left\{1,\frac{\pi(y)}{\pi(x)}\right\} q(y-x) \ud y
\end{multline*}
where $\charfun{A}$ stands for the characteristic function of the set $A$.
The proposal densities $\{q_s\}_{s\in\mathcal{C}^d}$ are defined as a mixture
\begin{equation}
    q_s(z) \defeq (1-\beta)\tilde{q}_s(z) + \beta \fixprop(z)
    \label{eq:mix-proposal}
\end{equation}
where the mixing constant $\beta\in[0,1)$ determines the portion how often a fixed
proposal density $\fixprop$ is used instead of the adaptive proposal 
$    \tilde{q}_s(z) \defeq \det(\theta s)^{-1/2}
\tilde{q}(\theta^{-1/2}s^{-1/2} z)$
with $\tilde{q}$ being a `template' probability density. 
Finally, the adaptation weights 
$(\eta_n)_{n\ge 2} \subset (0,1)$ appearing in \eqref{eq:m-rec} and
\eqref{eq:c-rec} is assumed to decay to zero.

One can verify that for $\beta=0$ this setting corresponds to the algorithm 
(S\ref{item:proposal})--(S\ref{item:adapt}) 
of Section
\ref{sec:intro} with $W_{n+1}$ having distribution $\tilde{q}$, and
for $\beta\in(0,1)$, (S\ref{item:proposal-mix}') applies instead of
(S\ref{item:proposal}).
Notice also that the original AM algorithm 
essentially fits this setting, with $\eta_n \defeq n^{-1}$, 
$\beta\defeq 0$ and if $\tilde{q}_s$ is defined slightly differently, 
being a Gaussian density with mean zero and covariance $s+\epsilon I$.
Moreover, if one sets $\beta=1$, the above setting 
reduces to a non-adaptive symmetric random walk Metropolis algorithm with
the increment proposal distribution $\fixprop$.

\section{The Unconstrained AM Algorithm} 
\label{sec:am} 

\subsection{Overview of the Results}
\label{sec:am-overview} 

This section deals with the unconstrained AM algorithm, that is, 
the algorithm described in Section \ref{sec:notations} with the mixing constant
$\beta = 0$ in \eqref{eq:mix-proposal}. 
Sections \ref{sec:expectation} and \ref{sec:path}
consider the case of an improper uniform target distribution $\pi \equiv c$ for some 
constant $c>0$. This implies that (almost) every proposed sample is accepted
and the recursion \eqref{eq:x-rec} reduces to 
\begin{equation}
    X_{n+1} = X_{n} + \theta S_n^{1/2} W_{n+1}
    \label{eq:adaptive-random-walk}
\end{equation}
where $(W_n)_{n\ge 2}$ are independent realisations of the distribution $\tilde{q}$.

Throughout this subsection, let us assume that the template proposal
distribution $\tilde{q}$ is spherically
symmetric and the weight sequence is defined as $\eta_n \defeq c n^{-\gamma}$ for
some constants $c\in(0,1]$ and $\gamma\in(1/2,1]$.
The first result characterises the expected behaviour of $S_n$
when $(X_n)_{n\ge 2}$ follows \eqref{eq:adaptive-random-walk}.
\begin{theorem} \label{thm:uncontrolled-asymptotics-simple} 
    Suppose $(X_n)_{n\ge 2}$ follows 
    the `adaptive random walk' recursion \eqref{eq:adaptive-random-walk},
    with $\E W_n W_{n}^T = I$.
    Then, for all $\lambda>1$ there is $n_0\ge m$ such that for
    all $n\ge n_0$ and $k\ge 1$, the following bounds hold
    \begin{equation*}
        \frac{1}{\lambda} \left(\theta\sum_{j=n+1}^{n+k} \sqrt{\eta_j}\right)
        \le \log\left(\frac{\Eop{S_{n+k}}}{\Eop{S_n}}\right)
        \le \lambda \left(\theta\sum_{j=n+1}^{n+k} \sqrt{\eta_j}\right).
    \end{equation*}
\end{theorem} 
\begin{proof}
Theorem \ref{thm:uncontrolled-asymptotics-simple} is a special case
of Theorem \ref{thm:uncontrolled-asymptotics} in Section \ref{sec:expectation}.
\end{proof}

\begin{remark}
Theorem \ref{thm:uncontrolled-asymptotics-simple} implies that
with the choice $\eta_n \defeq cn^{-\gamma}$
for some $c\in(0,1)$ and $\gamma\in(1/2,1]$,
the expectation grows with the speed
\[
    \Eop{S_n} \simeq
    \exp\left(\frac{\theta\sqrt{c}}{1-\frac{\gamma}{2}}n^{1-\frac{\gamma}{2}}\right).
\]
\end{remark}

\begin{remark} 
In the original setting \citep{saksman-am} the weights are defined as
$\eta_n\defeq n^{-1}$ and 
Theorem~\ref{thm:uncontrolled-asymptotics-simple}
implies that  the asymptotic growth rate of $\E[S_n]$ is
$e^{2\theta\sqrt{n}}$ when $(X_n)_{n\ge 2}$ follows
\eqref{eq:adaptive-random-walk}.
Suppose the value of $S_n$ is very small compared to the scale of a smooth target
distribution $\pi$. Then, it is expected that most of the proposal are
accepted, $X_n$ behaves almost as \eqref{eq:adaptive-random-walk}, and
$S_n$ is expected to grow approximately at the rate $e^{2\theta\sqrt{n}}$ until it 
reaches the correct magnitude. On the other
hand, simple deterministic bound implies that $S_n$ can decay slowly, only 
with the polynomial speed $n^{-1}$. Therefore, it may be safer to
choose the initial $s_1$ small.
\end{remark} 

\begin{remark} 
The selection of the scaling parameter $\theta>0$ in the AM algorithm 
does not seem to affect the expected asymptotic behaviour $S_n$ 
dramatically.
However, the choice $0<\theta\ll 1$ can  result
in an significant initial `dip' of the adapted covariance values,
as exemplified in Figure \ref{fig:dip}.
\begin{figure}
    \psfragscanon
    \psfrag{S}{$\Eop{S_n}$}
    \psfrag{n}{$n$}
    \includegraphics{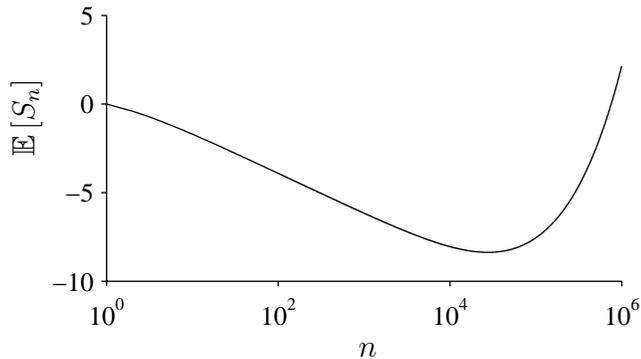}
    \caption{An example of the exact development of $\Eop{S_n}$, when $s_1 = 1$
     and $\theta=0.01$. The sequence $(\Eop{S_n})_{n\ge 1}$ decreases until $n$ is over
     $27,000$ and exceeds the initial value only with $n$ over $750,000$.}
    \label{fig:dip}
\end{figure}
Therefore, the
values $\theta\ll 1$ are to be used with care. In this case, the significance of a
successful burn-in is also emphasised.
\end{remark} 

It may seem that Theorem
\ref{thm:uncontrolled-asymptotics-simple} would automatically also ensure that
$S_n\to\infty$ also path-wise. This
is not, however, the case. For example, consider the probability
space $[0,1]$ with the Borel $\sigma$-algebra and the Lebesgue measure.
Then $(M_n,\F_n)_{n\ge 1}$ defined
as $M_n \defeq 2^{2n} \charfun{[0,2^{-n})}$ and 
$\F_n \defeq \sigma(X_k:1\le k\le n)$ is, in fact, a submartingale. Moreover,
$\E{M_n} = 2^n \to \infty$, but $M_n\to 0$ almost surely.

The AM process, however, does produce an unbounded sequence $S_n$.
\begin{theorem} 
    \label{thm:unif-path-simple} 
    Assume that $(X_n)_{n\ge 2}$ follows the `adaptive random walk'
    recursion \eqref{eq:adaptive-random-walk}.
    Then, for any unit vector $u\in\R^d$, the process $u^T S_n u \to \infty$
    almost surely. 
\end{theorem}
\begin{proof}
Theorem \ref{thm:unif-path-simple} is a special case of Theorem
\ref{thm:unif-path} in Section \ref{sec:path}.
\end{proof}

In a one-dimensional setting, and when $\log\pi$ is uniformly continuous, 
the AM process can be approximated with the `adaptive random walk' above,
whenever $S_n$ is small enough. This yields
\begin{theorem} 
    \label{thm:one-d-bound-simple} 
    Assume $d=1$ and $\log\pi$ is uniformly continuous.
    Then, there is 
    a constant $b>0$ such that $\liminf_{n\to\infty} S_n \ge b$.
\end{theorem} 
\begin{proof}
Theorem \ref{thm:one-d-bound-simple} is a special case of Theorem
\ref{thm:unif-path} in Section \ref{sec:example}.
\end{proof}

Finally, having Theorem \ref{thm:one-d-bound-simple}, it is possible to
establish 
\begin{theorem} 
    \label{thm:laplace-ergodicity-simple} 
    Assume $\tilde{q}$ is Gaussian, the one-dimensional
    target distribution is standard Laplace $\pi(x) \defeq \frac{1}{2}
    e^{-|x|}$ and the functional $f:\R\to\R$ satisfies
    $\sup_x e^{-\gamma|x|}|f(x)|<\infty$ for some $\gamma\in(0,1/2)$.
    Then, $n^{-1} \sum_{k=1}^n f(X_k) \to \int f(x) \pi(x) \ud x$
    almost surely as $n\to\infty$.
\end{theorem}
\begin{proof}
    Theorem \ref{thm:laplace-ergodicity-simple} is a special case of
    Theorem \ref{thm:laplace-ergodicity} in Section \ref{sec:example}.
\end{proof}
\begin{remark} 
    In the case
    $\eta_n \defeq n^{-1}$, Theorem \ref{thm:laplace-ergodicity-simple} implies
    that the parameters
    $M_n$ and $S_n$ of the adaptive chain converge to $0$ and $2$, that is,
    the true mean and variance of the target distribution $\pi$,
    respectively.
\end{remark}

\begin{remark} 
Theorem \ref{thm:one-d-bound-simple} (and Theorem
\ref{thm:laplace-ergodicity-simple}) could probably be extended to cover
also targets $\pi$ with compact supports. Such an extension would, however,
require specific handling of the boundary effects, which can lead to
technicalities.
\end{remark}

\subsection{Uniform Target: Expected Growth Rate} 
\label{sec:expectation} 

Define the
following matrix quantities
\begin{eqnarray}
    a_{n} &\defeq& \Eop{(X_{n}-M_{n-1})(X_{n}-M_{n-1})^T} \\
    b_{n} &\defeq& \Eop{S_{n}}
\end{eqnarray}
for $n\ge 1$, with the convention that $a_1 \equiv 0\in\R^{d\times d}$.
One may write
using \eqref{eq:m-rec} and \eqref{eq:adaptive-random-walk}
\[
    X_{n+1} - M_n = X_n - M_n + \theta S_n^{1/2} W_{n+1}
    = (1-\eta_n)(X_n - M_{n-1})+ \theta S_n^{1/2} W_{n+1}.
\]
If $\E W_n W_n^T = I$, one may easily compute
\begin{equation*}
    \begin{split}
    \E\big[&(X_{n+1}-M_{n})(X_{n+1}-M_{n})^T\big] \\
    &=\left(1-\eta_{n}\right)^2 \Eop{(X_n-M_{n-1})(X_n-M_{n-1})^T}
      +\theta^2\Eop{S_n }
\end{split}
\end{equation*}
since $W_{n+1}$ is independent of $\F_n$ and zero-mean due to the symmetry
of $\tilde{q}$. 
The values of $(a_n)_{n\ge 2}$ and $(b_n)_{n\ge 2}$ are therefore
determined by the joint recursion
\begin{eqnarray}
    a_{n+1} &=& (1-\eta_{n})^2 a_n + \theta^2 b_n
    \label{eq:a-recursion} \\
    b_{n+1} &=& (1-\eta_{n+1}) b_n + \eta_{n+1} a_{n+1}.
    \label{eq:b-recursion}
\end{eqnarray}
Observe that for any constant unit vector $u\in\R^d$, the recursions 
\eqref{eq:a-recursion} and \eqref{eq:b-recursion} hold also for
\begin{eqnarray*}
    a_{n+1}^{(u)} &\defeq& \Eop{u^T(X_{n+1}-M_n)(X_{n+1}-M_n)^T u} \\
    b_{n+1}^{(u)} &\defeq& \Eop{u^T S_{n+1} u}.
\end{eqnarray*}
The rest of this section therefore dedicates to the analysis if the one-dimensional 
recursions \eqref{eq:a-recursion} and \eqref{eq:b-recursion}, that is,
$a_n,b_n\in\R_+$ for all $n\ge 1$. The first result shows that the tail
of $(b_n)_{n\ge 1}$ is increasing.
\begin{lemma} \label{lemma:tail-increasing} 
    Let $n_0\ge 1$ and suppose $a_{n_0}\ge 0$, $b_{n_0}> 0$ and
    for $n\ge n_0$ the sequences $a_n$ and $b_n$ follow the recursions
    \eqref{eq:a-recursion} and \eqref{eq:b-recursion}, respectively.
    Then, there is a $m_0\ge n_0$ such that $(b_n)_{n\ge m_0}$ 
    is strictly increasing.
\end{lemma} 
\begin{proof} 
    If $\theta\ge 1$, we may estimate
    $a_{n+1} \ge (1-\eta_n)^2 a_n + b_n$ implying 
    $b_{n+1} \ge b_n + \eta_{n+1}(1-\eta_{n})^2 a_n$
    for all $n\ge n_0$.
    Since $b_n>0$ by construction, and therefore also
    $a_{n+1} \ge \theta^2 b_n>0$, we have that $b_{n+1} > b_n$ for 
    all $n\ge n_0+1$.
    
    Suppose then $\theta<1$.
    Solving $a_{n+1}$ from \eqref{eq:b-recursion} yields
\begin{equation*}
    a_{n+1} = \eta_{n+1}^{-1} \left(b_{n+1}-b_n\right) + b_n
\end{equation*}
Substituting this into \eqref{eq:a-recursion}, we obtain for $n\ge n_0+1$
\begin{equation*}
  \eta_{n+1}^{-1} \left(b_{n+1}-b_n\right) + b_n
  = (1-\eta_{n})^2 \left[
  \eta_{n}^{-1} \left(b_{n}-b_{n-1}\right) + b_{n-1}
  \right] + \theta^2 b_n
\end{equation*}
After some algebraic manipulation, this is equivalent to
\begin{equation}
    b_{n+1} - b_n
    = \frac{\eta_{n+1}}{\eta_n}(1-\eta_n)^3(b_n-b_{n-1}) +
  \eta_{n+1} \left[(1-\eta_n)^2 -1 + \theta^2 \right] b_n.
  \label{eq:b-alone-rec}
\end{equation}
Now, since $\eta_n\to 0$, we have that $(1-\eta_n)^2 -1 + \theta^2 >0$
whenever $n$ is greater than some $n_1$. So, if we
have for some $n'>n_1$ that $b_{n'}-b_{n'-1}\ge 0$, the sequence
$(b_n)_{n\ge n'}$ is strictly increasing after $n'$.

Suppose conversely that $b_{n+1}-b_{n}<0$ for all $n\ge n_1$. From
\eqref{eq:b-recursion}, $b_{n+1} - b_n = \eta_{n+1}(a_{n+1}-b_n)$
and hence $b_n>a_{n+1}$ for $n\ge n_1$. Consequently, 
from \eqref{eq:a-recursion}, $a_{n+1} > (1-\eta_n)^2 a_n + \theta^2
a_{n+1}$, which is equivalent to
\begin{equation*}
    a_{n+1} 
    > \frac{(1-\eta_n)^2}{1-\theta^2} a_n.
\end{equation*}
Since $\eta_n\to 0$, there is a $\mu>1$ and $n_2$ such that $a_{n+1} \ge \mu a_n$
for all $n\ge n_2$. That is, $(a_n)_{n\ge n_2}$ grows at least
geometrically, implying that after some time $a_{n+1}>b_n$, which is a
contradiction. To conclude, there is an $m_0\ge n_0$ such that
$(b_n)_{n\ge m_0}$ is strictly increasing.
\end{proof} 
Lemma \ref{lemma:tail-increasing} shows that the expectation $\Eop{u^T S_n
u}$ is ultimately bounded from below, assuming only that $\eta_n \to 0$.
By additional assumptions on the sequence $\eta_n$,
the growth rate can be characterised in terms of the adaptation weight sequence.
\begin{assumption}
    \label{a:adapt-weight} 
    Suppose $(\eta_n)_{n\ge 1}\subset(0,1)$ and
    there is $m'\ge 2$ such that 
    \begin{enumerate}[(i)]
    \item \label{item:adapt1} $(\eta_n)_{n\ge m'}$ is decreasing with
      $\eta_n\to 0$,
    \item \label{item:adapt3} $(\eta_{n+1}^{-1/2}-\eta_n^{-1/2})_{n\ge m'}$
      is decreasing and
    \item \label{item:adapt4} $\sum_{n=2}^\infty \eta_n = \infty$.
    \end{enumerate}
\end{assumption}
The canonical example of a sequence satisfying Assumption
\ref{a:adapt-weight} is the one assumed in Section
\ref{sec:am-overview}, $\eta_n \defeq cn^{-\gamma}$ for $c\in(0,1)$ and
$\gamma\in(1/2,1]$.
\begin{theorem} \label{thm:uncontrolled-asymptotics} 
    Suppose $a_{m}\ge 0$ and $b_{m}>
    0$ for some $m\ge 1$, and for $n> m$ the $a_n$ and $b_n$ are given
    recursively by \eqref{eq:a-recursion} and \eqref{eq:b-recursion},
    respectively. Suppose also that the sequence $(\eta_n)_{n\ge 2}$
    satisfies Assumption \ref{a:adapt-weight} with some $m'\ge m$.
    Then, for all $\lambda>1$ there is $m_2\ge m'$ such that for
    all $n\ge m_2$ and $k\ge 1$, the following bounds hold
    \begin{equation*}
        \frac{1}{\lambda} \left(\theta\sum_{j=n+1}^{n+k} \sqrt{\eta_j}\right)
        \le \log\left(\frac{b_{n+k}}{b_n}\right)
        \le \lambda \left(\theta\sum_{j=n+1}^{n+k} \sqrt{\eta_j}\right).
    \end{equation*}
\end{theorem} 
\begin{proof} 
    Let $m_0$ be the index from Lemma~\ref{lemma:tail-increasing} 
    after which the sequence $b_n$ is increasing.
Let $m_1> \max\{m_0,m'\}$ and
define the sequence $(z_n)_{n \ge m_1-1}$ by setting 
$z_{m_1-1} = b_{m_1-1}$ and
$z_{m_1} = b_{m_1}$, and for $n\ge m_1$ through the recursion
\begin{equation}
  z_{n+1} = z_n + \frac{\eta_{n+1}}{\eta_n}(1-\eta_n)^3(z_n-z_{n-1}) +
  \eta_{n+1}\tilde{\theta}^2 z_n
\end{equation}
where $\tilde{\theta}>0$ is a constant.
Consider such a sequence $(z_n)_{n\ge m_1-1}$ and define another sequence
$(g_n)_{n\ge m_1+1}$ through
\begin{equation*}
    \begin{split}
        g_{n+1} &\defeq 
        \eta_{n+1}^{-1/2}
      \frac{z_{n+1}-z_{n}}{z_{n}} 
      = 
      \eta_{n+1}^{-1/2} 
      \left[ 
\frac{\eta_{n+1}}{\eta_n}(1-\eta_n)^3\frac{z_n-z_{n-1}}{z_{n-1}}
\frac{z_{n-1}}{z_n} + \eta_{n+1}\tilde{\theta}^2
          \right] \\
      &= \eta_{n+1}^{1/2} \left(
      \frac{(1-\eta_n)^3}{\eta_n} \frac{g_n}{g_n + \eta_n^{-1/2}} +
       \tilde{\theta}^2\right).
\end{split}
\end{equation*}
Lemma \ref{lemma:g-convergence} below shows that $g_n \to \tilde{\theta}$.

Let us consider next two sequences $(z_n^{(1)})_{n\ge m_1-1}$ and 
$(z_n^{(2)})_{n\ge m_1-1}$ defined as $(z_n)_{n\ge m_1-1}$ above but using 
two different values $\tilde{\theta}^{(1)}$ and $\tilde{\theta}^{(2)}$,
respectively.
It is clear from \eqref{eq:b-alone-rec} that for the choice
$\tilde{\theta}^{(1)}\defeq \theta$ one has $b_n\le z_n^{(1)}$ for all $n\ge
m_1-1$.
Moreover, since $b_{m_1+1}/b_{m_1} \le z_{m_1+1}^{(1)}/z_{m_1}^{(1)}$, it holds 
by induction that
\begin{equation*}
    \begin{split}
    \frac{b_{n+1}}{b_n} &\le 1 +
    \frac{\eta_{n+1}}{\eta_n}(1-\eta_n)^3\left(1 - \frac{b_{n-1}}{b_n}\right)
    + \eta_{n+1}\tilde{\theta}^2 \\
    &\le 1 +
    \frac{\eta_{n+1}}{\eta_n}(1-\eta_n)^3\left(1 -
    \frac{z_{n-1}^{(1)}}{z_n^{(1)}}\right)
    + \eta_{n+1}\tilde{\theta}^2     
    = \frac{z_{n+1}^{(1)}}{z_n^{(1)}}
    \end{split}
\end{equation*}
also for all $n\ge m_1+1$.
By a similar argument one shows that if $\tilde{\theta}^{(2)} \defeq [
(1-\eta_{m_1})^2 -1 + \theta^2]^{1/2}$ then $b_n\ge z_n^{(2)}$ and $b_{n+1}/b_n\ge
z_{n+1}^{(2)}/z_n^{(2)}$ 
for all $n\ge m_1-1$.

Let $\lambda'>1$.
Since $g_n^{(1)}\to \tilde{\theta}^{(1)}$ and $g_n^{(2)} \to
\tilde{\theta}^{(2)}$ there is a $m_2\ge m_1$ such
that the following bounds apply
\begin{equation*}
    1 + \frac{\tilde{\theta}^{(2)}}{\lambda'}\sqrt{\eta_n} 
    \le
    \frac{z_n^{(2)}}{z_{n-1}^{(2)}}
    \qquad\text{and}\qquad
    \frac{z_n^{(1)}}{z_{n-1}^{(1)}}
    \le 1 + \lambda'
        \tilde{\theta}^{(1)}\sqrt{\eta_n} 
\end{equation*}
for all $n\ge m_2$.
Consequently, for all $n\ge m_2$, we have that
\begin{equation*}
    \log\left(\frac{b_{n+k}}{b_n}\right)
    \le \log\left(\frac{z_{n+k}^{(1)}}{z_n^{(1)}}\right) \le \sum_{j=n+1}^{n+k}
    \log\left(1+\lambda'\tilde{\theta}^{(1)}\sqrt{\eta_j}\right)
    \le \lambda'\theta\sum_{j=n+1}^{n+k}
    \sqrt{\eta_n}.
\end{equation*}
Similarly, by the mean value theorem 
\begin{equation*}
    \log\left(\frac{b_{n+k}}{b_n}\right)
    \ge \sum_{j=n+1}^{n+k}
    \log\left(1+\frac{\tilde{\theta}^{(2)}}{\lambda'}\sqrt{\eta_j}\right)
    \ge
    \frac{\tilde{\theta}^{(2)}}{\lambda'(1+\lambda'^{-1}\tilde{\theta}^{(2)}\sqrt{\eta_n})}
    \sum_{j=n+1}^{n+k} \sqrt{\eta_j}
\end{equation*}
since $\eta_n$ is decreasing. By letting the constant $m_1$ above be sufficiently large, 
the difference $|\tilde{\theta}^{(2)}-\theta|$ can be made arbitrarily
small, and by increasing $m_2$, the constant $\lambda'>1$ can be chosen
arbitrarily close to one.
\end{proof} 
Before Lemma \ref{lemma:g-convergence}, let us establish some properties of the
weight sequence $(\eta_n)_{n\ge 1}$ satisfying Assumption
\ref{a:adapt-weight}.
\begin{lemma} 
    \label{lemma:adapt-weight} 
    Suppose $(\eta_n)_{n\ge 1}$ satisfies Assumption \ref{a:adapt-weight}.
    Then,
    \begin{enumerate}[(a)]
    \item \label{item:adapt-c1}
      $(\eta_{n+1}/\eta_n)_{n\ge m'}$ is increasing with 
      $\eta_{n+1}/\eta_n \to 1$ and
    \item \label{item:adapt-c2} $\eta_{n+1}^{-1/2}-\eta_n^{-1/2} \to 0$.
    \end{enumerate}
\end{lemma}
\begin{proof} 
Define $a_n \defeq \eta_n^{-1/2}$ for all $n\ge m'$. By Assumption
\ref{a:adapt-weight} (\ref{item:adapt1}) 
$(a_n)_{n\ge m'}$ is increasing and
by Assumption
\ref{a:adapt-weight} (\ref{item:adapt3}), 
$(\Delta a_n)_{n\ge m'+1}$ is decreasing, where 
$\Delta a_n \defeq a_n - a_{n-1}$. One can write
\[
    \frac{a_{n}}{a_{n+1}} = 
    \frac{1}{1 + \frac{\Delta a_{n+1}}{a_n}}
    \ge \frac{1}{1 + \frac{\Delta a_{n}}{a_{n-1}}}
    = \frac{a_{n-1}}{a_{n}}
\]
implying that $(\eta_{n+1}/\eta_n)_{n\ge m'}$ is increasing. 
Denote $c = \lim_{n\to\infty} \eta_{n+1}/\eta_n\le 1$. It holds that 
$\eta_{m'+k} \le c \eta_{m'+k-1} \le \cdots \le c^k \eta_{m'}$. 
If $c<1$, then $\sum_n \eta_n < \infty$ contradicting
Assumption \ref{a:adapt-weight} (\ref{item:adapt4}), so $c$ must be one, establishing
(\ref{item:adapt-c1}).

From (\ref{item:adapt-c1}), one obtains
\[
    \frac{\eta_{n+1}^{-1/2} - \eta_n^{-1/2}}{\eta_n^{-1/2}} 
    = \left(\frac{\eta_n}{\eta_{n+1}}\right)^{1/2} -1
    \to 0
\]
implying (\ref{item:adapt-c2}).
\end{proof}

\begin{lemma} 
\label{lemma:g-convergence} 
Suppose $m_1\ge 1$, $g_{m_1}\ge 0$, the sequence $(\eta_n)_{n\ge m_1}$
satisfies Assumption \ref{a:adapt-weight} and $\tilde{\theta}>0$ is a constant. 
The sequence $(g_n)_{n> m_1}$ defined through
\begin{equation*}
    \begin{split}
        g_{n+1} \defeq \eta_{n+1}^{1/2} \left(
      \frac{(1-\eta_n)^3}{\eta_n} \frac{g_n}{g_n + \eta_n^{-1/2}} +
       \tilde{\theta}^2\right)
    \end{split}
\end{equation*}
satisfies $\lim_{n\to\infty} g_n = \tilde{\theta}$.
\end{lemma}
\begin{proof} 
Define the functions $f_n:\R_+\to\R_+$ for $n\ge m_1+1$ by
\begin{equation*}
    f_{n+1}(x) \defeq 
    \eta_{n+1}^{1/2} \left(
      \frac{(1-\eta_n)^3}{\eta_n} \frac{x}{x + \eta_n^{-1/2}} +
       \tilde{\theta}^2\right).
\end{equation*}
The functions $f_n$ are contractions on $[0,\infty)$ with contraction coefficient $q_n
\defeq (1-\eta_{n})^3$ since for all $x,y\ge 0$
\begin{equation*}\begin{split}
    \left|f_{n+1}(x)-f_{n+1}(y)\right|
    &
    = \eta_{n+1}^{1/2} 
    \frac{(1-\eta_n)^3}{\eta_n}
    \left|
       \frac{x}{x + \eta_n^{-1/2}} -
      \frac{y}{y + \eta_n^{-1/2}} 
       \right| \\
    &= \left(\frac{\eta_{n+1}}{\eta_n}\right)^{1/2} 
    \frac{(1-\eta_n)^3}{\eta_n}
    \left|
       \frac{x-y}{(x + \eta_n^{-1/2})(y + \eta_n^{-1/2})}
       \right| \\
    &\le \left(\frac{\eta_{n+1}}{\eta_n}\right)^{1/2} 
    (1-\eta_n)^3
    \left| x-y \right| \le q_{n+1} \left| x-y \right|
    \end{split}
\end{equation*}
where the second inequality holds since $\eta_{n+1}\le \eta_n$.

The fixed point of $f_{n+1}$ can be written as
\begin{equation*}
    x^*_{n+1} \defeq \frac{1}{2}\left(
    -\xi_{n+1} + \sqrt{\xi_{n+1}^2 +
 \mu_{n+1}}
    \right)
\end{equation*}
where 
\begin{eqnarray*}
    \xi_{n+1} &\defeq &
    \eta_n^{-1/2} - \eta_{n+1}^{1/2}\eta_n^{-1}
(1-\eta_{n})^3 -
    \eta_{n+1}^{1/2}\tilde{\theta}^2 \\
    \mu_{n+1} &\defeq&
4 \eta_n^{-1/2}\eta_{n+1}^{1/2} \tilde{\theta}^2.
\end{eqnarray*}
Lemma \ref{lemma:adapt-weight} (\ref{item:adapt-c1}) implies
$\mu_{n+1} \to 4 \tilde{\theta}^2$. Moreover,
\begin{equation*}
    \begin{split}
    \xi_{n+1} &=
    \eta_n^{-1/2} - \eta_{n+1}^{1/2}\eta_n^{-1} 
    + \eta_{n+1}^{1/2}(3 - 3\eta_n + \eta_n^2
    -\tilde{\theta}^2)\\
    &= \left(\frac{\eta_{n+1}}{\eta_n}\right)^{1/2}
    \left(\eta_{n+1}^{-1/2}-\eta_n^{-1/2}\right)
    + \eta_{n+1}^{1/2}(3 - 3\eta_n + \eta_n^2 
    -\tilde{\theta}^2).
    \end{split}
\end{equation*}
Therefore, 
by Assumption \ref{a:adapt-weight} (\ref{item:adapt1}) and Lemma
\ref{lemma:adapt-weight}, $\xi_{n+1} \to 0$ and consequently
the fixed points satisfy $x_n^* \to \tilde{\theta}$.

Consider next the consecutive differences of the fixed points.
Using the mean value theorem and the triangle inequality, write
\begin{equation*}
    \begin{split}
    2\left|x_{n+1}^* - x_n^*\right| 
    &\le
    \left| \xi_{n+1}-\xi_n  \right|
      + \frac{1}{2\sqrt{\tau_n}}
      \left|
      \xi_{n+1}^2 -
      \xi_n^2
      + \mu_{n+1} - \mu_{n}
      \right| \\
    &\le 
\left| \xi_{n+1}-\xi_n  \right|      
      + \frac{\tau_n'}{\sqrt{\tau_n}}
\left| \xi_{n+1}-\xi_n  \right|
      + \frac{1}{2\sqrt{\tau_n}}
      \left|\mu_{n+1} - \mu_{n}
      \right| \\
    &\le c_1 
\left| \xi_{n+1}-\xi_n  \right|
      + c_1\left|\mu_{n+1} - \mu_{n}\right|
    \end{split}
\end{equation*}
where the value of ${\tau_n}$ is between 
$\xi_{n+1}^2 + \mu_{n+1}$ and 
$\xi_n^2 + \mu_n$ converging to $4\tilde{\theta}^2>0$, 
the value of $\tau_n'$ is between
$|\xi_{n+1}|$ and $|\xi_n|$ converging to zero,
and $c_1>0$ is a constant.

The differences of the latter terms satisfy for all $m\ge m'$ 
\begin{equation*}
    \begin{split}
    \sum_{n=m'}^m
    \left|\mu_{n+1} - \mu_{n}\right|
        &= 4\tilde{\theta}^2 \sum_{n=m'}^m
        \left[
            \left(\frac{\eta_{n+1}}{\eta_{n}}\right)^{1/2}
            - \left(\frac{\eta_n}{\eta_{n-1}}\right)^{1/2}
        \right] \\
        &\le 4\tilde{\theta}^2 
        \left[ 1 -
            \left(\frac{\eta_{m'}}{\eta_{m'-1}}\right)^{1/2}
        \right] \le 4 \tilde{\theta}^2.
        \end{split}
\end{equation*}
by Assumption \ref{a:adapt-weight} (\ref{item:adapt3}) and Lemma
\ref{lemma:adapt-weight} (\ref{item:adapt-c1}).
For the first term, let us estimate
\begin{equation*}
    \begin{split}
    \left|\xi_{n+1}-\xi_n\right| 
    \le \left|
    \left(\frac{\eta_{n+1}}{\eta_n}\right)^{1/2}
    \left(\eta_{n+1}^{-1/2}-\eta_n^{-1/2}\right)
    - \left(\frac{\eta_{n}}{\eta_{n-1}}\right)^{1/2}
    \left(\eta_{n}^{-1/2}-\eta_{n-1}^{-1/2}\right)
    \right| \\
    + \left|3-\tilde{\theta}^2\right|\left|
       \eta_{n}^{1/2}- \eta_{n+1}^{1/2}\right|
    + \left|\eta_{n+1}^{1/2}(3\eta_n - \eta_n^2)
    - \eta_{n}^{1/2}(3\eta_{n-1} - \eta_{n-1}^2) \right|.
    \\
    \end{split} 
\end{equation*}
Assumption \ref{a:adapt-weight} (\ref{item:adapt1}) implies that
$\eta_{n}^{1/2}-\eta_{n+1}^{1/2}\ge 0$
for $n\ge m'$ and hence $\sum_{n=m'}^m \left|\eta_{n}^{1/2}- \eta_{n+1}^{1/2}\right|
    \le \eta_{m'}^{1/2}$ for any $m\ge m'$. Since
the function $(x,y)\mapsto x(3y-y^2)$ is Lipschitz on $[0,1]^2$,
there is a constant $c_2$ independent of $n$ such that
$
    \big|\eta_{n+1}^{1/2}(3\eta_n - \eta_n^2)
    - \eta_{n}^{1/2}(3\eta_{n-1} - \eta_{n-1}^2) \big|
    \le c_2 \big(
    | \eta_{n+1}^{1/2} - \eta_n^{1/2} | 
    + | \eta_n - \eta_{n-1} | \big)
$, and a similar argument shows that
\[
    \sum_{n=m'}^m \left|\eta_{n+1}^{1/2}(3\eta_n - \eta_n^2)
    - \eta_{n}^{1/2}(3\eta_{n-1} - \eta_{n-1}^2) \right| \le c_3 <\infty.
\]
One can also estimate
\begin{equation*}
    \begin{split}
        &
    \left|
    \left(\frac{\eta_{n+1}}{\eta_n}\right)^{1/2}
    \left(\eta_{n+1}^{-1/2}-\eta_n^{-1/2}\right)
    - \left(\frac{\eta_{n}}{\eta_{n-1}}\right)^{1/2}
    \left(\eta_{n}^{-1/2}-\eta_{n-1}^{-1/2}\right)
    \right| \\
    &\le 
    c_4
    \left|
    \left(\frac{\eta_{n+1}}{\eta_n}\right)^{1/2}
    -\left(\frac{\eta_{n}}{\eta_{n-1}}\right)^{1/2}
    \right|
    + c_4\left|
    \left(\eta_{n+1}^{-1/2}-\eta_n^{-1/2}\right)-
    \left(\eta_{n}^{-1/2}-\eta_{n-1}^{-1/2}\right)
    \right|
    \end{split}
\end{equation*}
yielding by 
Assumption \ref{a:adapt-weight}
(\ref{item:adapt3}) 
and Lemma \ref{lemma:adapt-weight} that
$ \sum_{n=m'}^m | \xi_{n+1}-\xi_n | \le c_5$
for all $m\ge m'$, with a constant $c_5<\infty$.
Combining the above estimates, the fixed point differences satisfy
\begin{equation*}
    \sum_{n=m'}^m |x^*_{n+1} - x^*_{n}| < \infty.
\end{equation*}

Fix a $\delta>0$ and let $n_\delta>m_1$ be sufficiently large so that 
$\sum_{k=n_\delta+1}^\infty |x^*_{n+1} - x^*_{n}| \le \delta$ implying
also that
$|x^*_n-\tilde{\theta}|\le\delta$ for all $n\ge n_\delta$.
Then, for $n\ge n_\delta$ one may write
\begin{equation*}
    \begin{split}
    \left|g_n - \tilde{\theta}\right| 
        &\le \left|g_n -x_{n}^*\right| 
          + \left|x_n^* - \tilde{\theta}\right| 
        \le \left|f_n(g_{n-1}) - f_n(x_{n}^*)\right| 
          + \delta
       \\
        &\le q_n\left|g_{n-1} - x_{n}^*\right| 
          + \delta
        \le q_n\left|g_{n-1} - x_{n-1}^*\right| 
          + \left|x_{n-1}^* - x_{n}^*\right| 
          + \delta
        \\
        &\le q_n q_{n-1}\left|g_{n-2} - x_{n-2}^*\right| 
          + \left|x_{n-2}^* - x_{n-1}^*\right| 
          + \left|x_{n-1}^* - x_{n}^*\right| 
          + \delta \\
        &\le \cdots 
        \le \left(\prod_{k=n_\delta+1}^n q_k\right) 
        \left|g_{n_\delta} - x_{n_\delta}^*\right|
            + 2\delta.
   \end{split}
\end{equation*}
Since
$    \log \prod_{k=n_\delta+1}^n q_k
= 3 \sum_{k=n_\delta+1}^n \log(1-\eta_{k-1})
\le - 3 \sum_{k=n_\delta}^{n-1} \eta_k \to -\infty
$ as $n\to\infty$ by Assumption \ref{a:adapt-weight} (\ref{item:adapt4}),
it holds that $(\prod_{k=n_\delta+1}^n
q_k)|g_{n_\delta}-x_{n_\delta}^*|\to 0$.
That is, $|g_n - \tilde{\theta}| \le 3\delta$
for any sufficiently large $n$, and since $\delta>0$ was arbitrary,
$g_n\to \tilde{\theta}$. 
\end{proof}


\subsection{Uniform Target: Path-wise Behaviour} 
\label{sec:path} 

Section \ref{sec:expectation} characterised the behaviour of the sequence 
$\Eop{S_{n}}$ when the chain $(X_n)_{n\ge 2}$ follows the `adaptive
random walk' recursion \eqref{eq:adaptive-random-walk}.
In this section, we shall verify that almost every sample path $(S_n)_{n\ge
1}$ of the same process are increasing.

Fix a unit vector $u\in\R^d$ and define the scalar process
$(Z_n)_{n\ge 2}$ through
\begin{equation}
    Z_{n+1} \defeq u^T \frac{X_{n+1}-M_n}{\|S_n^{1/2} u\|}.
    \label{eq:z-def}
\end{equation}
where $\|x\| \defeq \sqrt{x^T x}$ stands for the Euclidean norm.
The behaviour of the process $(Z_n)_{n\ge 2}$ determines the behaviour of
$(u^T S_n u)_{n\ge 2}$ since one can write a recursion for
$(u^T S_n u)_{n\ge 2}$ using only $(Z_n)_{n\ge 2}$
\begin{equation}
    \begin{split}
    u^T S_{n+1} u &= (1-\eta_{n+1})u^T S_n u +
    \eta_{n+1}u^T(X_{n+1}-M_n)(X_{n+1}-M_n)^T u \\
    &= [1 + \eta_{n+1} (Z_{n+1}^2-1)] u^T S_n u.
    \end{split}
    \label{eq:z-char-s}
\end{equation}
On the other hand, one can express $(Z_n)_{n\ge 2}$ in terms of
$(W_n)_{n\ge 2}$ and $(S_n)_{n\ge 1}$
\[
    \begin{split}
    Z_{n+1} & = \theta u^T \frac{S_n^{1/2} W_{n+1}}{\| S_n^{1/2} u \|} 
    + (1-\eta_n)u^T \frac{X_n - M_{n-1}}{\| S_n^{1/2} u \|} \\
    &= \theta \frac{u^T S_n^{1/2} }{\|S_n^{1/2} u\|} W_{n+1}
    + (1-\eta_n)\left(\frac{u^T S_{n-1} u}{u^T S_n u}\right)^{1/2} Z_n.
    \end{split}
\]
Using \eqref{eq:z-char-s}, this simplifies to
\begin{equation}
    Z_{n+1}= \theta \tilde{W}_{n+1} + U_n Z_n
    \label{eq:z-almost-random-walk}
\end{equation}
where 
\[
    \tilde{W}_{n+1} \defeq  \frac{u^T S_n^{1/2} }{\|S_n^{1/2} u\|} W_{n+1}
    \qquad\text{and}\qquad
    U_n \defeq (1-\eta_n)\left(\frac{1}{1+\eta_n(Z_n^2 - 1)}\right)^{1/2}.
\]
Let us observe first that $(\tilde{W}_{n})_{n\ge 2}$ are independent if
the distribution $\tilde{q}$ of $(W_n)_{n\ge 2}$ is spherically symmetric.
\begin{lemma} 
    Assume $(W_n)_{n\ge 1}$ are independent and follow a spherically
    symmetric non-degenerate distribution in $\R^d$.  Then
    $(\tilde{W}_n)_{n\ge 1}$ are independent and identically distributed
    non-degenerate real-valued random variables.
\end{lemma}
\begin{proof} 
    Choose a measurable $A\subset\R$,
    denote $T_n \defeq \|S_n^{1/2} u\|^{-1}S^{1/2}_n u$ and
    define $A_n \defeq \{x\in\R^d: T_n^T x \in A\}$. Let $R_n$ be a rotation
    matrix such that $R_n^T T_n = e_1\defeq (1,0,\ldots,0)\in\R^d$.
    Since $W_{n+1}$ is independent of $\F_n$, we have
    \[\begin{split}
        \P(\tilde{W}_{n+1}\in A\mid\F_n)
        &= \P(W_{n+1}\in A_n\mid \F_{n})
        = \P(R_n W_{n+1}\in A_n\mid \F_{n}) \\
        &= \P(e_1^T W_{n+1} \in A\mid \F_n) 
        = \P(e_1^T W_1 \in A) 
        \end{split}
    \]
    by the rotational invariance of the distribution of $(W_{n})_{n\ge 1}$.
\end{proof} 
Notice particularly that if $({W}_n)_{n\ge 2}$  are standard Gaussian
vectors in $\R^d$ then $(\tilde{W}_n)_{n\ge 2}$ are standard Gaussian random
variables.

Only values $|Z_{n}|<1$ can decrease
$S_{n}$ as shown by \eqref{eq:z-char-s}. 
But if both $\eta_n$ and $\eta_n Z_n^2$ are small, the variable $U_n$ is
clearly close to unity, and consequently $Z_n$ behaves almost as a random walk.
Let us consider an auxiliary result quantifying the behaviour of this
random walk.
\begin{lemma} 
    \label{lemma:hit-tube} 
Let $n_0\ge 2$, suppose 
$\tilde{Z}_{n_0-1}$ is $\F_{n_0-1}$-measurable random variable 
and suppose $(\tilde{W}_n)_{n\ge n_0}$ are respectively
$(\F_n)_{n\ge n_0}$-measurable and non-degenerate i.i.d.~random variables.
Define for $\tilde{Z}_n$ for $n\ge 2$ through
\[
    \tilde{Z}_{n+1} = \tilde{Z}_n + \theta \tilde{W}_{n+1}.
\]
Then, for any $N,\delta_1,\delta_2 > 0$, there is a $k_0\ge 1$ such that
\[
    \Pcond{
      \frac{1}{k} \sum_{j=1}^{k} \charfun{\{|\tilde{Z}_{n+j}|\le N\}}
      \ge \delta_1
    }{ \F_n } \le \delta_2
\]
a.s.~for all $n\ge 1$ and $k\ge k_0$.
\end{lemma} 
\begin{proof} 
From the Kolmogorov-Rogozin inequality, Theorem
\ref{th:kolmogorov-rogozin} in Appendix \ref{sec:kolmogorov-rogozin},
\[
    \P(\tilde{Z}_{n+j}-\tilde{Z}_{n} \in [x,x+2N] \mid \F_n) \le c_1 j^{-1/2}
\]
for any $x\in\R$, where the constant $c_1>0$ depends on $N$, $\theta$ and 
on the distribution of $W_j$. In particular, since $\tilde{Z}_{n+j}-\tilde{Z}_{n}$ is
independent of $\tilde{Z}_{n}$, one may set $x=-Z_n-N$ above, and thus
$\Pcond{|\tilde{Z}_{n+j}|\le N}{\F_n}\le c_1 j^{-1/2}$.
The estimate
\[
    \Econd{
      \frac{1}{k} \sum_{j=1}^{k} \charfun{\{|\tilde{Z}_{n+j}|\le N\}}
    }{ \F_n } \le \frac{c_1}{k} \sum_{j=1}^{k} j^{-1/2} \le c_2 k^{-1/2}
\]
implies $\P\big(k^{-1} \sum_{j=1}^{k} \charfun{\{|\tilde{Z}_{n+j}|\le N\}} \ge
\delta_1\;\big|\;\F_n\big) \le \delta_1^{-1} c_2 k^{-1/2}$,
concluding the proof.
\end{proof} 
The technical estimate in the next Lemma \ref{lemma:drift-bound} 
makes use of the above mentioned random walk approximation and
guarantees ultimately a positive `drift' for
the eigenvalues of $S_n$. The result requires that the
adaptation sequence $(\eta_n)_{n\ge 2}$ is `smooth' in the sense 
that the quotients converge to zero.
\begin{assumption}
    \label{a:adapt-nonosc} 
    The adaptation weight sequence $(\eta_n)_{n\ge 2}\subset(0,1)$ satisfies
    \[
        \lim_{n\to\infty} \frac{\eta_{n+1}}{\eta_n} = 1.
    \]
\end{assumption} 
\begin{lemma}
    \label{lemma:drift-bound} 
Let $n_0\ge 2$, 
suppose $Z_{n_0-1}$ is $\F_{n_0-1}$-measurable, 
and assume $(Z_n)_{n\ge n_0}$ follows \eqref{eq:z-almost-random-walk} with
non-degenerate i.i.d.~variables
$(\tilde{W}_n)_{n\ge n_0}$ measurable with respect to $(\F_n)_{n\ge
n_0}$, respectively, and
the adaptation weights $(\eta_n)_{n\ge n_0}$ satisfy Assumption
\ref{a:adapt-nonosc}. Then, for any $C\ge 1$ and $\epsilon>0$, 
there are indices $k\ge 1$ 
and $n_1\ge n_0$ such that $\Pcond{L_{n,k}}{\F_n} \le \epsilon$ 
a.s.~for all $n\ge n_1$, where
\[
    L_{n,k} \defeq \left\{
      \sum_{j=1}^{k} \log\left[1 + \eta_{n+j}\left(Z_{n+j}^2 - 1\right)\right]
    < k C \eta_n
    \right\}.
\]
\end{lemma}
\begin{proof} 
Fix $\gamma\in(0,2/3)$ and assume
$Z_n^2 \le \eta_n^{-\gamma}$. One may estimate
\begin{eqnarray*}
    U_n  & = &
    (1-\eta_n)^{1/2}
    \left(1 - \frac{\eta_nZ_n^2}{1-\eta_n+\eta_n Z_n^2}\right)^{1/2} \\
    &\ge& (1-\eta_n)^{1/2}
    \left(1 - \frac{\eta_n^{1-\gamma}}{1-\eta_n}\right)^{1/2} \\
    &\ge& (1-\eta_n^{1-\gamma})^{1/2}
    \left(\frac{1 - 2\eta_n^{1-\gamma}}{1-\eta_n}\right)^{1/2}
    \ge 1 - c_1\eta_n^{1-\gamma}
\end{eqnarray*}
where $c_1\defeq 2\sup_{n\ge n_0} (1-\eta_n)^{-1/2}<\infty$.
Observe also that $U_n \le 1$.

Let $k_0\ge 1$ be from Lemma \ref{lemma:hit-tube} applied with
$N = \sqrt{8C + 1}+1$,
$\delta_1=1/8$ and
$\delta_2=\epsilon$, and fix $k\ge k_0+1$.
Let $n\ge n_0$ and
define an auxiliary process $(\tilde{Z}_j^{(n)})_{j\ge n_0-1}$ as
$\tilde{Z}_{j}^{(n)} \equiv Z_j$ for $n_0-1\le j\le n+1$, and for $j>n+1$ through
\[
    \tilde{Z}_{j}^{(n)} = Z_{n+1} + \theta \sum_{i=n+2}^j \tilde{W}_i.
\]
For any $n+2\le j\le n+k$ and 
$\omega \in A_{n:j} \defeq \cap_{i=n+1}^{j} \{ Z_i^2\le
  \eta_i^{-\gamma}\}$,
the difference of $\tilde{Z}_j^{(n)}$ and $Z_j$ can be bounded by
\[\begin{split}
    |\tilde{Z}_{j+1}^{(n)} - Z_{j+1}|
    &\le |Z_{j}||1-U_j| + |\tilde{Z}_{j}^{(n)} - Z_{j}|
    \le c_1 \eta_j^{1 - \frac{3}{2}\gamma} + |\tilde{Z}_{j}^{(n)} - Z_{j}| 
    \le \cdots \\
    &\le c_1\sum_{i=n+1}^j \eta_i^{1-\frac{3}{2}\gamma}
    \le c_1\eta_{n}^{1-\frac{3}{2}\gamma}\sum_{i=n+1}^j
    \left(\frac{\eta_i}{\eta_{n}}\right)^{1-\frac{3}{2}\gamma}
    \le c_2 (j-n)\eta_{n}^{1-\frac{3}{2}\gamma}
\end{split}\]
by Assumption \ref{a:adapt-nonosc}.
Therefore, for sufficiently large $n\ge n_0$, the inequality
$|\tilde{Z}_{j}^{(n)} - Z_{j}| \le 1$ holds for all $n\le j\le n+k$
and $\omega\in A_{n:n+k}$. Now, if $\omega\in A_{n:n+k}$, the following bound holds
\[
    \begin{split}
    \log&\left[1 + \eta_j(Z_j^2 - 1)\right]
    \ge
    \log\left[1 + \eta_j(\min\{N,|Z_j|\}^2 - 1)\right]\\
    &\ge \charfun{\{|\tilde{Z}_j^{(n)}|> N\}}\log\left[1 + \eta_j((N-1)^2-1)\right] 
    + \charfun{\{|\tilde{Z}_j^{(n)}|\le N\}}\log\left[1-\eta_j\right] \\
    &\ge \charfun{\{|\tilde{Z}_j^{(n)}|> N\}}(1-\beta_j) \eta_j 8C - 
    \charfun{\{|\tilde{Z}_j^{(n)}|\le N\}}(1+\beta_j)\eta_j
    \end{split}
\]
by the mean value theorem, where the constant 
$\beta_j=\beta_j(C,\eta_j)\in(0,1)$ can be selected arbitrarily small
whenever $j$ is sufficiently large. Using this estimate, one can write
for $\omega\in A_{n:n+k}$
\[
\sum_{j=1}^{k} \log\left[1 + \eta_{n+j}\left(Z_{n+j}^2 - 1\right)\right] 
\ge (1-\beta_n)\sum_{j\in I_{n+1:k}^+} \eta_{n+j}8C
- (1+\beta_n)\sum_{j=1}^k  \eta_{n+j} \\
\]
where
$I_{n+1:k}^+ \defeq \{j\in[1,k]: \tilde{Z}_{n+j}^{(n)}> N\}$.
Define the sets 
\[
    B_{n,k} \defeq \left\{\frac{1}{k-1}\sum_{j=1}^{k-1}
  \charfun{\{|\tilde{Z}_{n+j+1}|\le N\}}\le \delta_1\right\}. 
\]
Within $B_{n,k}$, it clearly holds that $\# I_{n+1:k}^+ \ge k - 1 -
(k-1)\delta_1 = 7(k-1)/8$. 
Thereby, for all $\omega\in
B_{n,k}\cap A_{n:n+k}$
\begin{multline*}
\sum_{j=1}^{k} \log\left[1 + \eta_{n+j}\left(Z_{n+j}^2 - 1\right)\right]
\\
 \ge \eta_n k \left[ (1-\beta_n)
 \frac{7}{2}
 \left(\inf_{1\le j\le k}\frac{\eta_{n+j}}{\eta_n}\right)
 C   
 - (1+\beta_n) 
 \left(\sup_{1\le j\le k} \frac{\eta_{n+j}}{\eta_n} \right)
\right] \ge k C \eta_n
\end{multline*}
for sufficiently large $n\ge 1$, as then the constant $\beta_n$ can be
chosen small enough, and by Assumption \ref{a:adapt-nonosc}.
In other words, if $n\ge 1$ is sufficiently large, then 
$B_{n,k}\cap A_{n:n+k}\cap L_{n,k} = \emptyset$.

Let us then write the conditional expectation of interest in parts,
\begin{equation}
    \begin{split}
    \Pcond{L_{n,k}}{\F_n}
    &= \Pcond{L_{n,k}, A_{n:n+k}}{\F_n}
       + \Pcond{L_{n,k}, A'_n}{\F_n} \\
    &\phantom{=}+ \sum_{i=n+1}^{n+k} \Pcond{L_{n,k}, A_{n:i-1}, A'_i}{\F_n}
    \end{split}
    \label{eq:l-estim}
\end{equation}
where $A'_i \defeq \{ Z_i^2>\eta_i^{-\gamma}\}$.
Let $\omega\in A'_i$ for any $n<i\le n+k$ and compute
\[\begin{split}
    \log\left[1 + \eta_i(Z_i^2 - 1)\right]
    &\ge \log\left[1 + \eta_i(\eta_i^{-\gamma} - 1)\right]
    \ge \log\left[1 + 2\eta_i kC\right] \\
    &\ge \frac{2\eta_ikC}{1+2\eta_ikC}
    \ge k C \eta_n
    \end{split}
\]
whenever $n\ge n_0$ is sufficiently large, since $\eta_n\to 0$,
and by Assumption \ref{a:adapt-nonosc}.
That is, if $n$ is sufficiently large, 
all but the first term in the right hand side of 
\eqref{eq:l-estim} are a.s.~zero.
It remains to show the inequality for the first, for which the estimate
\[\begin{split}
    \Pcond{L_{n,k}, A_{n:n+k}}{\F_n}
    &=\Pcond{L_{n,k}, A_{n:n+k}, B_{n,k}}{\F_n} \\
    &\phantom{=}+
    \Pcond{L_{n,k}, A_{n:n+k}, \smash{B_{n,k}^\complement}}{\F_n} \\
    &\le \Pcond{\smash{B_{n,k}^\complement}}{\F_n} \le \epsilon
\end{split}\]
holds by Lemma \ref{lemma:hit-tube},
concluding the proof.
\end{proof}

Using the estimate of Lemma \ref{lemma:drift-bound}, it is relatively easy 
to show that the eigenvalues of $S_n$ tend to infinity, if the adaptation
weights satisfy an additional assumption.
\begin{assumption}
    \label{a:adapt-ell2-not-ell1} 
    The adaptation weight sequence $(\eta_n)_{n\ge 2}\subset(0,1)$ 
    is in $\ell^2$ but not in $\ell^1$,
    that is,
    \[
        \sum_{n=2}^\infty \eta_n = \infty 
        \quad\text{and}\quad 
        \sum_{n=2}^\infty \eta_n^2 < \infty.
    \]
\end{assumption}
\begin{theorem} 
    \label{thm:unif-path} 
    Assume that $(X_n)_{n\ge 2}$ follows the `adaptive random walk'
    recursion \eqref{eq:adaptive-random-walk} and 
    the adaptation weights $(\eta_n)_{n\ge 2}$ 
    satisfy Assumptions \ref{a:adapt-nonosc}
    and \ref{a:adapt-ell2-not-ell1}.
    Then, for any unit vector $u\in\R^d$, the process $u^T S_n u \to \infty$
    almost surely. 
\end{theorem}
\begin{proof} 
    The proof is based on the estimate of Lemma \ref{lemma:drift-bound}
    applied with a similar martingale argument as in
    \cite{vihola-asm}.
    
    Let $k\ge 2$ be from Lemma \ref{lemma:drift-bound} applied with $C=4$ and
    $\epsilon = 1/2$. Denote $\ell_i \defeq ki + 1$ for $i\ge 0$ and,
    inspired by \eqref{eq:z-char-s},
    define the random variables $(T_i)_{i\ge 1}$ by
    \[
        T_i \defeq \min\Bigg\{kM\eta_{\ell_{i-1}},
          \sum_{j=\ell_{i-1}+1}^{\ell_{i}} 
           \log\left[1 + \eta_{j}\left(Z_{j}^2 - 1\right)\right]
  \Bigg\}
    \]
    with the convention that $\eta_0 = 1$.
    Form a martingale $(Y_i, \G_i)_{i\ge 1}$ with $Y_1\equiv 0$ and having
    differences $\ud Y_i \defeq T_i - \Econd{T_i}{\G_{i-1}}$ and where
    $\G_1 \equiv \{\emptyset,\Omega\}$ and $\G_i \defeq \F_{\ell_i}$ for $i\ge 1$.
    By Assumption \ref{a:adapt-ell2-not-ell1},
    \[
        \sum_{i=2}^\infty \Eop{\ud Y_i^2} \le 
        c \sum_{i=1}^\infty \eta_{\ell_i}^2
        < \infty
    \]
    with a constant $c=c(k,C)>0$, 
    so $Y_i$ is a $L^2$-martingale and converges a.s.\ to a finite limit $M_\infty$
    \cite[e.g.][Theorem 2.15]{hall-heyde}.
    
    By Lemma \ref{lemma:drift-bound}, the conditional expectation satisfies
    \[
        \Econd{T_{i+1}}{\mathcal{G}_{i}} \ge
        k C \eta_{\ell_i}(1-\epsilon)
        + \sum_{j=\ell_i+1}^{\ell_{i+1}} \log(1-\eta_{j})\epsilon
        \ge k\eta_{\ell_i}
    \]
    when $i$ is large enough, and where the second inequality is due to 
    Assumption \ref{a:adapt-nonosc}.
    This implies, with Assumption
    \ref{a:adapt-ell2-not-ell1}, that
    $\sum_{i} \Econd{T_i}{\mathcal{G}_{i-1}} = \infty$ a.s., and since
    $Y_i$ converges a.s.~to a finite limit, it holds that 
    $\sum_i T_i = \infty$ a.s.

    By \eqref{eq:z-char-s}, one may estimate for any $n = \ell_m$ with $m\ge 1$ that
    \[
        \log(u^T S_n u) \ge \log(u^T S_1 u) + \sum_{i=1}^m T_i
        \to \infty
    \]
    as $m\to\infty$.
    Simple deterministic estimates
    conclude the proof for the intermediate values of $n$.
\end{proof}


\subsection{Stability with One-Dimensional Uniformly Continuous Log-Density}
\label{sec:example} 

In this section, the above analysis of the `adaptive random walk' is
extended to imply that $\liminf_{n\to\infty} S_n>0$ 
for the one-dimensional AM algorithm, assuming $\log\pi$ uniformly
continuous. The result follows similarly as in Theorem
\ref{thm:unif-path}, by coupling the AM process with the `adaptive random
walk' whenever $S_n$ is smaller than some constant $\mu>0$.
\begin{theorem} 
    \label{thm:one-d-bound} 
    Assume $d=1$ and $\log\pi$ is uniformly continuous, and that
    the adaptation weights $(\eta_n)_{n\ge 2}$ satisfy Assumptions
    \ref{a:adapt-nonosc} and \ref{a:adapt-ell2-not-ell1}.
    Then, there is 
    a constant $b>0$ such that $\liminf_{n\to\infty} S_n \ge b$.
\end{theorem} 
\begin{proof} 
    Fix a $\delta\in(0,1)$.
    Due to the uniform continuity of $\log\pi$, there is a
    $\tilde{\delta}>0$ such that
    \[
        \log \pi(y) - \log \pi(x) \ge \frac{1}{2}\log\left(1 -
        \frac{\delta}{2}\right)
    \]
    for all $|x-y|\le \tilde{\delta}_1$. Choose 
    $\tilde{M}>0$ sufficiently large so that 
    $\int_{\{|z|\le \tilde{M}\}} \tilde{q}(z) \ud
    z \ge
    \sqrt{1-\delta/2}$. Denote by
    \[
        Q_q(x,A) \defeq \int_A q(y-x) \ud y
    \]
    the random walk transition kernel with increment distribution $q$, and observe
    that the `adaptive random walk' recursion \eqref{eq:adaptive-random-walk} 
    can be written as ``$X_{n+1} \sim Q_{\tilde{q}_{S_n}}(X_n,\uarg)$.'' 
    For any $x\in\R^d$ and measurable $A\subset\R^d$
    \[
        \begin{split}
        |Q_{\tilde{q}_s}(x,A) - P_{\tilde{q}_s}(x,A)| &\le 2 \left[ 1 -
        \int\min\left\{1,\frac{\pi(y)}{\pi(x)}\right\} \tilde{q}_s(y-x) \ud y \right]
        \\
        & \le 
        2 \left[ 1 -
        \int_{\{|z|\le \tilde{M}\}} 
        \min\left\{1,\frac{\pi(x+\smash{\sqrt{\theta s}}z)}{\pi(x)}\right\}
         \tilde{q}(z)
        \ud z \right].
        \end{split}
    \]
    Now, $|Q_{\tilde{q}_s}(x,A) - P_{\tilde{q}_s}(x,A)\| \le \delta$
    whenever $\sqrt{\theta s} z \le \tilde{\delta_1}$ for all $|z|\le
    \tilde{M}$. In other words, there exists a
    $\mu=\mu(\delta)>0$ such that whenever $s<\mu$, the total variation norm 
    $\|Q_{\tilde{q}_s}(x,\cdot) - P_{\tilde{q}_s}(x,\cdot)\| \le \delta$.
    
    Let $n,k\ge 1$ and
    define the random variables 
    $(\tilde{X}_{j}^{(n)},\tilde{M}_j^{(n)},\tilde{S}_j^{(n)})_{j\in[n,n+k]}$
    by setting 
    $(\tilde{X}_{n}^{(n)}, \tilde{M}_{n}^{(n)},
    \tilde{S}_{n}^{(n)} ) \equiv (X_{n},
    M_{n},S_{n})$
    and 
    \begin{eqnarray*}
        \tilde{X}_{j+1}^{(n)} &\sim&
        Q_{\tilde{q}_{\tilde{S}_j^{(n)}}}(\tilde{X}_j^{(n)},\uarg),\\
        \tilde{M}_{j+1}^{(n)} &\defeq&
        (1-\eta_{j+1})\tilde{M}_j^{(n)} + \eta_{j+1}
        \tilde{X}_{j+1}^{(n)} \quad\text{and} \\
        \tilde{S}_{j+1}^{(n)} &\defeq&
        (1-\eta_{j+1})\tilde{S}_j^{(n)} + \eta_{j+1}
        (\tilde{X}_{j+1}^{(n)}-\tilde{M}_j^{(n)})^2
    \end{eqnarray*}
    for $j+1\in[n+1,n+k]$.
    The variable $\tilde{X}_{n+1}^{(n)}$ can be selected so that
    $\P(\tilde{X}_{n+1}^{(n)}
      = X_{n+1} \mid \F_n) = 1 - \|
P_{\tilde{q}_{S_n}}(X_n,\uarg)-Q_{\tilde{q}_{\tilde{S}_n^{(n)}}}(\tilde{X}_n^{(n)},\uarg)\|$; 
    see Theorem \ref{thm:coupling} in Appendix \ref{sec:coupling}.
    Consequently,
    $\P(\tilde{X}_{n+1}^{(n)} \neq X_{n+1},\,S_n<
        \mu\mid\F_n) \le
    \delta$.
    By the same argument, $\tilde{X}_{n+2}^{(n)}$ can be chosen so that
    \[
        \P\big(\tilde{X}_{n+2}^{(n)} \neq X_{n+2},\,
        \tilde{X}_{n+1}^{(n)} = X_{n+1},\, S_{n+1}<\mu
        \;\big|\;\sigma(\F_{n+1},\tilde{X}_{n+1}^{(n)})\big) \le \delta
    \]
    since if $\tilde{X}_{n+1}^{(n)} = X_{n+1}$, then also
    $\tilde{S}_{n+1}^{(n)} = S_{n+1}$. This implies
    \begin{equation*}
        \P\big(\{\tilde{X}_{n+2}^{(n)} \neq X_{n+2}\} \cup
        \{\tilde{X}_{n+1}^{(n)} \neq X_{n+1}\}\cap
        B_{n:n+2}\;\big|\;\F_n\big) \le
        2\delta
    \end{equation*}
    where $B_{n:j} \defeq \cap_{i=n}^{j-1} \{S_i<\mu\}$ for $j>n$.
    The same argument can be repeated 
    to construct $(\tilde{X}_j^{(n)})_{j\in[n,n+k]}$ 
    so that 
    \begin{equation}
        \Pcond{ D_{n:n+k} }{\F_{n}} \ge 1 - k\delta
        \label{eq:d-est}
    \end{equation}
    where $D_{n:n+k} \defeq \bigcap_{j=n}^{n+k} 
    \{\tilde{X}_j^{(n)} = X_j \} \cup B_{n:n+k}^\complement$.

    Apply Lemma \ref{lemma:drift-bound} 
    with $C=18$ and
    $\epsilon=1/6$ to obtain $k\ge 1$, and fix $\delta=\epsilon/k$.
    Denote $\ell_i \defeq ik +1$ for any $i\ge 0$, and
    define the random variables $(T_i)_{i\ge 1}$ by
    \begin{equation}
        T_{i} \defeq 
\charfun{\{S_{\ell_{i-1}}<\mu/2\}}
        \min\Bigg\{kM\eta_{\ell_{i-1}},
          \sum_{j=\ell_{i-1}+1}^{\ell_{i}} 
           \log\left[1 + \eta_{j}\left(Z_{j}^2 - 1\right)\right]    
  \Bigg\}
    \label{eq:ti-def}
    \end{equation}
    where $Z_j$ are defined as
    \eqref{eq:z-def}. 

    Define also $\tilde{T}_i$ similarly as $T_i$, but having 
    $\tilde{Z}_{j}^{(\ell_{i-1})}$ with $j\in[\ell_{i-1}+1,\ell_{i}]$ in the
    right hand side of \eqref{eq:ti-def}, defined as
    $\tilde{Z}_{\ell_{i-1}}^{(\ell_{i-1})} \equiv Z_{\ell_{i-1}}$ and by
    \[
    \tilde{Z}_{j}^{(\ell_{i-1})} \defeq \big(
\tilde{X}_{j}^{(\ell_{i-1})}-\tilde{M}_{j-1}^{(\ell_{i-1})}\big)\Big/
\sqrt{\tilde{S}_{j-1}^{(\ell_{i-1})}}.
    \]
    for $j\in[\ell_{i-1}+1,\ell_{i}]$.
    Notice that $T_{i}$ coincides
    with $\tilde{T}_{i}$ in $B_{\ell_{i-1}:\ell_{i}}\cap
    D_{\ell_{i-1}:\ell_{i}}$. Observe also that
    $\tilde{X}_{j}^{(\ell_{i-1})}$ follows the `adaptive random
    walk' equation \eqref{eq:adaptive-random-walk} for
    $j\in[\ell_{i-1}+1,\ell_{i}]$, and
    hence $\tilde{Z}_{j}^{(\ell_{i-1})}$ follows
    \eqref{eq:z-almost-random-walk}. 
    Consequently, denoting $\G_i \defeq F_{\ell_i}$,
    Lemma \ref{lemma:drift-bound} guarantees that
    \begin{equation}
        \Pcond{L_{\ell_{i-1},k}}{\G_i} \le \epsilon
        \label{eq:l-est}
    \end{equation}
    where $L_{\ell_{i-1},k} \defeq
    \{\tilde{T}_{i} < k M \eta_{\ell_{i-1}}\}$.
    
    Let us show next that whenever $S_{\ell_{i-1}}$ is small, 
    the variable $T_i$ is expected to have a positive value proportional to
    the adaptation weight,
    \begin{equation}
        \Econd{T_{i}}{\G_{i-1}}\charfun{\{S_{\ell_{i-1}}<\mu/2\}}
    \ge k\eta_{\ell_{i-1}}\charfun{\{S_{\ell_{i-1}}<\mu/2\}}
    \label{eq:t-est}
    \end{equation} 
    almost surely
    for any sufficiently large $i\ge 1$.
    Write first
    \begin{multline*}
        \Econd{T_{i}}{\G_{i-1}}\charfun{\{S_{\ell_{i-1}}<\mu/2\}}
        = \Econd{
          (\charfun{B_{\ell_{i-1}:\ell_{i}}^\complement} +
          \charfun{B_{\ell_{i-1}:\ell_{i}}})
          T_{i}}{\G_{i-1}}\charfun{\{S_{\ell_{i-1}}<\mu/2\}} \\
         \ge \Econd{
\charfun{B_{\ell_{i-1}:\ell_{i}}^\complement}\min\left\{kC\eta_{\ell_{i-1}},\frac{\mu}{2}+
                    \xi_{i}\right\}
                    + \charfun{B_{\ell_{i-1}:\ell_{i}}} \xi_{i}
                    }{\G_{i-1}}
                  \charfun{\{S_{\ell_{i-1}}<\mu/2\}}
     \end{multline*}
where the lower bound $\xi_{i}$ of $T_{i}$ is given as
    \[
        \xi_{i} \defeq \sum_{j=\ell_{i-1}+1}^{\ell_{i}} \log(1-\eta_{j}).
    \]
By Assumption \ref{a:adapt-nonosc}, $\xi_{i} \ge - 2k\eta_{\ell_{i-1}} \ge
-\mu/4$ for any sufficiently large $i$.  Therefore, whenever
$\Pcond{B_{\ell_{i-1}:\ell_{i}}^\complement}{\G_{i-1}} \ge \epsilon = 3/C$,
it holds that
    \[
        \Econd{T_{i}}{\G_{i-1}}\charfun{\{S_{\ell_{i-1}}<\mu/2\}}
        \ge k\eta_{\ell_{i-1}} \charfun{\{S_{\ell_{i-1}}<\mu/2\}}
    \]
for any sufficiently large $i$.
On the other hand, if $\Pcond{B_{\ell_{i-1}:\ell_{i}}^\complement}{\G_{i-1}}
\le \epsilon$, then by defining
\[
    E_i \defeq B_{\ell_{i-1}:\ell_{i}}^\complement \cup
    D_{\ell_{i-1}:\ell_{i}}^\complement \cup L_{\ell_{i-1},k}
\]
one has by \eqref{eq:d-est} and \eqref{eq:l-est} that $\P(E_i)\le 3\epsilon$,
and consequently
    \[\begin{split}
        \Econd{T_{i}}{\G_{i-1}}
        &\ge
        \Pcond{ E_i^\complement}{\G_{i-1}}\xi_{i} 
        + \Econd{ \charfun{E_i}
          \tilde{T}_{i}}{\G_{i-1}} \\
        &\ge
        3\epsilon \xi_{i} + (1-3\epsilon)kC\eta_{\ell_{i-1}}
        \ge k\eta_{\ell_{i-1}}.
        \end{split}
    \]
This establishes \eqref{eq:t-est}.

Define the stopping times $\tau_1 \equiv 1$ and for $n\ge 2$ through
$\tau_n \defeq \inf\{i>\tau_{n-1}: S_{\ell_{i-1}}\ge \mu/2,\, S_{\ell_i}<\mu/2\}$
with the convention that $\inf\emptyset = \infty$.
That is, $\tau_i$ record the times when $S_{\ell_i}$
enters $(0,\mu/2]$.  Using $\tau_i$, define the latest such time up to $n$ by
$\sigma_n \defeq \sup\{\tau_i: i\ge 1,\,\tau_i \le n\}$.
As in Theorem
\ref{thm:unif-path}, define the almost surely converging 
martingale $(Y_i, \G_i)_{i\ge 1}$ 
with $Y_1\equiv 0$ and having the 
differences $\ud  Y_i \defeq (T_i - \Econd{T_i}{\G_{i-1}})$
for $i\ge 2$.

It is sufficient to show that $\liminf_{i\to\infty}
S_{\ell_i} \ge b \defeq \mu/4>0$ almost surely.
If there is a finite $i_0\ge 1$ such that $S_{\ell_i} \ge \mu/2$ for all
$i\ge i_0$, the claim is trivial.
Let us consider for the rest of the proof the case 
that $\{S_{\ell_i} < \mu/2\}$ happens for infinitely many indices $i\ge 1$.

For any $m\ge 2$ such that $S_{\ell_m} < \mu/2$,
one can write
\begin{equation}
    \begin{split}
    \log S_{\ell_m} &\ge \log S_{\ell_{\sigma_m} } + \sum_{i=\sigma_m+1}^m T_i \\
    &\ge \log S_{\ell_{\sigma_m}} + (Y_m-Y_{\sigma_m}) + \sum_{i=\sigma_m+1}^m
    k\eta_{\ell_{i-1} }
    \end{split}
    \label{eq:s-ineq}
\end{equation}
since then $S_{\ell_i}<\mu/2$ for all $i\in[\sigma_m,m-1]$
and hence also $\Econd{T_{i}}{\G_{i-1}} \ge k\eta_{\ell_{i-1}}$.

Suppose for a moment that there is a positive probability that 
$S_{\ell_m}$ stays within $(0,\mu/2)$ indefinitely, starting from some
index $m_1\ge 1$. 
Then, there is an infinite $\tau_i$ and consequently
$\sigma_m\le \sigma<\infty$ for all $m\ge 1$. But as $Y_m$ converges,
$|Y_m-Y_{\sigma_m}|$ is a.s.~finite, and since $\sum_m \eta_{\ell_m} = \infty$
by Assumptions \ref{a:adapt-nonosc} and \ref{a:adapt-ell2-not-ell1},
the inequality \eqref{eq:s-ineq} implies that $S_{\ell_m} \ge \mu/2$ for
sufficiently large $m$, which is a contradiction. That is, the stopping
times $\tau_i$ for all $i\ge 1$ must be a.s.~finite, 
whenever $S_{\ell_m}<\mu/2$ for infinitely many indices $m\ge 1$.

For the rest of the proof, suppose $S_{\ell_m}<\mu/2$ for infinitely many
indices $m\ge 1$.
Observe that since
$Y_m\to Y_\infty$, there exists an a.s.~finite index $m_2$ so that $Y_m-Y_\infty
\ge -1/2\log 2$ for all $m\ge m_2$. As $\eta_n\to 0$ and $\sigma_m\to\infty$, 
there is an a.s.~finite $m_3$ such that
$\xi_{\sigma_{m-1}} \ge -1/2 \log 2$ for all $m\ge m_3$. For all $m\ge
\max\{m_2,m_3\}$ and whenever $S_{\ell_m}< \mu/2$,
it thereby holds that
\begin{equation*}
    \begin{split}
    \log S_{\ell_m} 
    &\ge \log S_{\ell_{\sigma_{m}}} - (Y_m - Y_{\sigma_m})
    \ge \log S_{\ell_{\sigma_{m-1}}} + \xi_{\sigma_m}
    - \frac{1}{2}\log 2 \\
    &\ge \log \frac{\mu}{2} - \log 2 = \log b.
    \end{split}
\end{equation*}
The case $S_{\ell_m}\ge \mu/2$ trivially satisfies the above estimate,
concluding the proof.
\end{proof} 
As a consequence of Theorem \ref{thm:one-d-bound}, one can establish
a strong law of large numbers for the unconstrained AM algorithm running
with a Laplace target distribution.
Essentially, the only ingredient that needs to be checked is that the
simultaneous geometric ergodicity condition holds. This is 
verified in the next lemma, whose proof is given in Appendix
\ref{sec:proof-exp-geombound}.
\begin{lemma} 
    \label{lemma:exp-geombound} 
    Suppose that
    the template proposal distribution $\tilde{q}$ is everywhere positive
    and non-increasing away from the origin: $\tilde{q}(z)\ge \tilde{q}(w)$ for all 
    $|z|\le|w|$.
    Suppose also that $\pi(x) \defeq 
    \frac{1}{2b}\exp\left(-\frac{|x-m|}{b}\right)$ with a mean $m\in\R$
    and a scale $b>0$.
    Then, 
    for all $L>0$, there are positive constants $M,b$
such that the following drift and minorisation condition are satisfied
for all $s\ge L$ and measurable $A\subset \R$
\begin{align}
    P_s V(x) &\le \lambda_s V(x)  + b \charfun{C}(x), 
    & \forall x\in\R\label{eq:drift-ineq} \\
    P_s(x,A) &\ge \delta_s \nu(A), & \forall x\in C\label{eq:mino-ineq}
\end{align}
where $V:\R\to[1,\infty)$ is defined as
$V(x) \defeq (\sup_z \pi(z))^{1/2} \pi^{-1/2}(x)$, 
the set $C \defeq [m-M,m+M]$, the probability measure $\mu$ is 
concentrated on $C$ and $P_s V(x) \defeq \int V(y) P_s(x,\ud y)$. 
Moreover, $\lambda_s,\delta_s\in(0,1)$ satisfy for all $s\ge L$
\begin{equation}
        \max\{
    (1-\lambda_s)^{-1},
    \delta_s^{-1}\}
    \le c s^{\gamma}
    \label{eq:const-bound}
\end{equation} 
for some constants $c,\gamma>0$ that may depend on $L$.
\end{lemma}
\begin{theorem} 
    \label{thm:laplace-ergodicity} 
    Assume the adaptation weights $(\eta_n)_{n\ge 2}$ satisfy Assumptions
    \ref{a:adapt-nonosc} and \ref{a:adapt-ell2-not-ell1},
    and the template proposal density $\tilde{q}$ and the target
    distribution $\pi$ satisfy the assumptions in Lemma
    \ref{lemma:exp-geombound}.
    If the functional $f$ satisfies $\sup_{x\in\R}
    \pi^{-\gamma}(x)|f(x)|<\infty$ for some $\gamma\in(0,1/2)$.
    Then, $n^{-1} \sum_{k=1}^n f(X_k) \to \int f(x) \pi(x) \ud x$
    almost surely as $n\to\infty$.
\end{theorem}
\begin{proof} 
The conditions of \ref{thm:one-d-bound} are clearly satisfied
implying that for any $\epsilon>0$ there is a 
$\kappa=\kappa(\epsilon)>0$ such that the event
\[
    B_{\kappa} \defeq \left\{\inf_n S_n  \ge
\kappa\right\}
\]
has a probability $\P(B_{\kappa})\ge 1-\epsilon$. 
    
The inequalities \eqref{eq:drift-ineq} and \eqref{eq:mino-ineq} of Lemma
\ref{lemma:exp-geombound} with the
bound \eqref{eq:const-bound} imply, using 
\cite[Proposition 10 and Lemma 15]{saksman-vihola}, that for any
$\beta>0$ there is a constant $A=A(\kappa,\epsilon,\beta)<\infty$ such that
$\P(B_\kappa\cap\{\max\{|S_n|,|M_n|\} > A n^\beta\})\le \epsilon$. 
Let us define the sequence of truncation sets 
\[
    K_n \defeq
\{(m,s)\in\R\times \R_+: \lambda_{\min}(s)\ge \kappa,\,
  \max\{|s|,|m|\}\le A n^\beta \}
\]
for $n\ge 1$. Construct an auxiliary truncated process
$(\tilde{X}_n,\tilde{M}_n,\tilde{S}_n)_{n\ge 1}$, starting from
$(\tilde{X}_1,\tilde{M}_1,\tilde{S}_1) \equiv (X_1,M_1,S_1)$ and for $n\ge
2$ through
\begin{eqnarray*}
    \tilde{X}_{n+1} & \sim & P_{\tilde{q}_{\tilde{S}_n}}(\tilde{X}_n,\cdot) \\
    (\tilde{M}_{n+1},\tilde{S}_{n+1}) & = &
    \sigma_{n+1}\Big[(\tilde{M}_n,\tilde{S}_n),\,  
    \eta_{n+1}\big(\tilde{X}_{n+1}-\tilde{M}_n,
(\tilde{X}_{n+1}-\tilde{M}_n)^2-\tilde{S}_n\big)\Big]
\end{eqnarray*}
where the truncation function $\sigma_{n+1}:(K_n)\times (\R\times\R)\to K_n$ is defined as
\[
    \sigma_{n+1}(z,z') = \begin{cases} z + z',&\text{if $z+z'\in K_n$} \\
                                       z, &\text{otherwise}.
                                       \end{cases}
\]
Observe that this constrained process coincides with the AM process with
probability
$\P\big(\forall n\ge 1:(\tilde{X}_n,\tilde{M}_n,\tilde{S}_n)=(X_n,M_n,S_n)\big)
\ge 1-2\epsilon$.
Moreover, \cite[Theorem 2]{saksman-vihola} implies that a strong law of
large numbers holds for the truncated process $(\tilde{X}_n)_{n\ge 1}$,
since $\sup_x |f(x)|V^{-\alpha}(x)<\infty$ for some $\alpha\in(0,1-\beta)$,
by selecting $\beta>0$ above sufficiently small.
Since $\epsilon>0$ was arbitrary, the strong law of large
numbers holds for $(X_n)_{n\ge 1}$.
\end{proof}



\section{AM With a Fixed Proposal Component} 
\label{sec:fcam} 

This section deals with the modification due to Roberts and Rosenthal
\cite{roberts-rosenthal-examples}, including a fixed component in the
proposal distribution. In terms of Section \ref{sec:notations}, the
mixing parameter in \eqref{eq:mix-proposal} 
satisfies $0<\beta<1$. Theorem \ref{th:am-zerobound}
shows that the fixed proposal component guarantees, with a 
verifiable non-restrictive Assumption \ref{a:uniform-component}, that 
the eigenvalues of the adapted covariance parameter $S_n$ are bounded
away from zero. As in Section \ref{sec:example}, this result implies an ergodicity 
result, Theorem \ref{th:ergodicity}. 

Let us start by formulating the key assumption that, intuitively speaking, 
assures that the adaptive
chain $(X_n)_{n\ge 1}$ will have `uniform mobility' regardless of the
adaptation parameter $s\in\mathcal{C}^d$.
\begin{assumption} 
    \label{a:uniform-component} 
    There exist a compactly supported probability measure $\nu$ 
    that is absolutely continuous with respect to the Lebesgue measure,
    constants $\delta>0$ and $c<\infty$ and a measurable mapping 
    $\xi:\R^d \times \mathcal{C}^d \to\R^d$
    such that for all
    $x\in\R^d$ and $s\in\mathcal{C}^d$, 
    \begin{equation*}
        \|\xi(x,s) - x \| \le c
        \qquad\text{and}\qquad
        P_{q_s}(x,A) \ge \delta \nu\big(A-\xi(x,s)\big) 
    \end{equation*}
    for all measurable sets $A\subset\R^d$, where $A-y \defeq \{x-y:x\in
      A\}$ is the translation of the set $A$ by $y\in\R^d$.
\end{assumption} 
\begin{remark} 
    \label{rem:uniform-continuous-component} 
    In the case of the AM algorithm with a fixed proposal component,
    one is primarily interested in the case where
    $\xi(x,s) = \xi(x)$ and for all $x\in\R^d$ 
    \[
        \beta \fixprop(x-y) \min\left\{1, \frac{\pi(y)}{\pi(x)}\right\}
        \ge \delta \nu\big(y-\xi(x)\big)
    \]
    for all $y\in\R^d$,
    where $\nu$ is a uniform density on some ball.
    Then, 
    since $P_{q_s} = (1-\beta) P_{\tilde{q}_s} + \beta P_{\fixprop}$, 
    \[
        P_{q_s}(x,A) \ge \beta P_{\fixprop}(x,A) \ge \delta \int_A \nu(y-\xi) \ud y
    \]
    and Assumption \ref{a:uniform-component} is fulfilled by the measure
    $\nu(A) \defeq \int_A \nu(y) \ud y$. 
\end{remark} 

Having Assumption \ref{a:uniform-component}, the lower bound on the
eigenvalues of $S_n$ can be obtained relatively easily, by a martingale argument
similar to the one used in Section \ref{sec:am} and in \cite{vihola-asm}.
\begin{theorem} 
    \label{th:am-zerobound} 
Let $(X_n,M_n,S_n)_{n\ge 1}$ be an AM process as defined in Section
\ref{sec:notations} satisfying Assumption \ref{a:uniform-component}.
Moreover, suppose that the adaptation weights $(\eta_n)_{n\ge 2}$ satisfy
Assumptions \ref{a:adapt-nonosc} and \ref{a:adapt-ell2-not-ell1}.
Then, 
\[
    \liminf_{n\to\infty} \inf_{w\in\mathcal{S}^d} w^T S_n w > 0
\]
where $\mathcal{S}^d$ stands for the unit sphere.
\end{theorem} 
\begin{proof} 
Let us first introduce independent binary auxiliary variables $(Z_n)_{n\ge 2}$ 
with $Z_1 \equiv 0$, and through
\begin{eqnarray*}
    \Pcond{Z_{n+1}=1}{X_n,M_n,S_n,Z_n} &=& \delta \\
    \Pcond{Z_{n+1}=0}{X_n,M_n,S_n,Z_n} &=& (1-\delta).
\end{eqnarray*}
Using this auxiliary variable, we can assume $X_{n}$ to 
follow\footnote{by possibly augmenting the probability space; see 
  \citep{athreya-ney,nummelin-splitting}.}
\begin{equation*}
    X_{n+1}  =  Z_{n+1}(U_{n+1}+\Xi_n)
      + (1-Z_{n+1})R_{n+1}
\end{equation*}
where $U_{n+1}\sim \nu(\cdot)$ is independent of $\F_n$ and $Z_{n+1}$,
the random variable $\Xi_n \defeq \xi(X_n,S_n)$ is $\F_n$-measurable, and 
$R_{n+1}$ is distributed according to the `residual' transition kernel
$\check{P}_{S_n}(X_n,A) \defeq (1-\delta)^{-1}[P_{q_{S_n}}(X_n,A) -
\delta\nu(A-\Xi_n)]$, valid by Assumption \ref{a:uniform-component}.

Define
$\mathcal{S}(w,\gamma) \defeq \{v\in\mathcal{S}^d: \|w-v\|\le
\gamma\}$, the segment of the unit 
sphere centred at $w\in\mathcal{S}^d$ and having the radius $\gamma>0$.
Fix a unit vector $w\in\mathcal{S}^d$ and 
define the following random variables
\[
    \begin{split}
    \Gamma_{n+2}^{(\gamma)}
    &\defeq
    \inf_{v\in\mathcal{S}(w,\gamma)}
      \left(
        |v^T(X_{n+1}-M_n)|^2 + |v^T(X_{n+2}-M_{n+1})|^2
      \right) 
   \end{split}
\]
for all $n\ge 1$.
Denote $G_{n+1} \defeq X_{n+1}-M_n$ and $E_{n+1} \defeq \Xi_{n+1}-X_{n+1}$, 
and observe that whenever $Z_{n+2}=1$, it holds that
\[
    X_{n+2}-M_{n+1} = U_{n+2} + X_{n+1} - M_{n+1} + E_{n+1}
    = U_{n+2} + (1-\eta_{n+1})G_{n+1} + E_{n+1}
\]
and we may write
\[
    Z_{n+2} \Gamma_{n+2}^{(\gamma)}
    = Z_{n+2} \inf_{v\in\mathcal{S}(w,\gamma)}
      \left(
        |v^T G_{n+1}|^2 + |v^T(U_{n+2} + \lambda_{n+1} G_{n+1} + E_{n+1})|^2
      \right) 
\]
where $\lambda_{n} \defeq 1-\eta_{n}\in(0,1)$ for all $n\ge 2$.
Consequently, we may apply Lemma \ref{lemma:eigenbound} below
to find constants $\gamma,\mu>0$ such that
\begin{equation}
    \Pcond{Z_{n+2} \Gamma_{n+2}^{(\gamma)}\ge \mu}{\F_n}
    \ge \frac{\delta}{2}.
    \label{eq:gamma-estim}
\end{equation}
Hereafter, assume $\gamma>0$ is fixed such that \eqref{eq:gamma-estim} holds, and
denote $\Gamma_{n+2} \defeq \Gamma_{n+2}^{(\gamma)}$
and $\mathcal{S}(w)\defeq \mathcal{S}(w,\gamma)$.

Consider the random variables
\begin{multline}
    D_{n+2} \defeq 
    \inf_{v\in\mathcal{S}(w)}
    \big(\eta_{n+1} |v^T(X_{n+1}-M_n)|^2
    + \eta_{n+2} |v^T(X_{n+2}-M_{n+1})|^2\big) \\
    \ge \min\{\eta_{n+1},\eta_{n+2}\}
    \Gamma_{n+2} 
    \ge \eta_* \eta_{n+1} \Gamma_{n+2}
    \label{eq:gamma-estim-sdiff}
\end{multline}
where $\eta_* \defeq \inf_{k\ge 2}\eta_{k+1}/\eta_k>0$
by Assumption \ref{a:adapt-nonosc}.
Define the indices $\ell_n \defeq 2n-1$ for $n\ge 1$
and let 
\[
    T_{n} \defeq \eta_*\min\{\mu,Z_{\ell_n}\Gamma_{\ell_n}\}
\]
for all $n\ge 2$.
Define the $\sigma$-algebras $\G_n \defeq \F_{\ell_n}$ for $n\ge 1$ and
observe that $\Econd{T_{n+1}}{\G_n} \ge \eta_*\mu\delta/2$ by
\eqref{eq:gamma-estim}.
Construct a  martingale starting from $Y_1 \equiv 0$ and having the
differences $\ud Y_{n+1} \defeq \eta_{\ell_n+1}(T_{n+1} -
\Econd{T_{n+1}}{\G_n})$.
The martingale $Y_n$ converges to an a.s.~finite limit $Y_\infty$ as 
in Theorem \ref{thm:unif-path}.

Define also $\eta^*\defeq \sup_{k\ge 2} \eta_{k+1}/\eta_k<\infty$ and
$\kappa \defeq \inf_{k\ge 2} 1 - \eta_k>0$, and
let 
\[
    b\defeq \frac{\kappa\eta_*\mu\delta}{8\eta^*}  > 0.
\]
Denote
$S_{n}^{(w)} \defeq \inf_{v\in\mathcal{S}(w)} v^T S_{n} v$ 
and define the stopping times $\tau_1 \equiv 1$ and for $k\ge 2$ through
\[
    \tau_k \defeq \inf \{n > \tau_{k-1}: S_{\ell_n}^{(w)} \le b,\, 
      S_{\ell_{n-1}}^{(w)} > b\}
\]
with the convention
$\inf\emptyset = \infty$. That is, $\tau_k$ record the times when
$S_{\ell_n}^{(w)}$
enters $(0,b]$.  Using $\tau_k$, define the latest such time up to $n$ by
$\sigma_n \defeq \sup\{\tau_k: k\ge 1,\,\tau_k \le n\}$.

Observe that
for any $n\ge 2$ such that $S_{\ell_n}^{(w)} \le b$,
one may write
\begin{equation*}
    \begin{split}
    S_{\ell_n}^{(w)} 
    &= S_{\ell_{\sigma_n}}^{(w)} + \sum_{k=\sigma_n}^{n-1}
    \left( D_{\ell_k+2} - \eta_{\ell_k+1}
        S_{\ell_k}^{(w)} -
\eta_{\ell_k+2} S_{\ell_k+1}^{(w)}
    \right) \\
    &\ge S_{\ell_{\sigma_n}}^{(w)} + \sum_{k=\sigma_n}^{n-1}
    \left(\eta_{\ell_k+1} T_{k+1} 
    - \eta_{\ell_k+1}b -
        \eta_{\ell_k+2}\kappa^{-1}b
    \right) \\
    &\ge S_{\ell_{\sigma_n}}^{(w)} + \sum_{k=\sigma_n}^{n-1}
    \eta_{\ell_k+1} \left(T_{k+1} 
    - \frac{\eta_*\mu\delta}{4}
    \right)
    \end{split}
\end{equation*}
by \eqref{eq:gamma-estim-sdiff} and
since for all $k\in[\sigma_n,n-1]$ one may estimate
$S_{\ell_k +1}^{(w)} \le (1-\eta_{\ell_k + 1})^{-1} 
S_{\ell_{k+1}}^{(w)} \le \kappa^{-1} b$.

That is, for any $n\ge 2$ such that $S_{\ell_n}^{(w)} \le b$
\begin{equation*}
    \begin{split}
    S_{\ell_n}^{(w)} 
    & \ge S_{\ell_{\sigma_n}}^{(w)} + (Y_n-Y_{\sigma_n})
      + \sum_{k=\sigma_n}^{n-1}
    \eta_{\ell_k+1}\left( \Econd{T_{k+1}}{\G_k} - \frac{\eta_*\delta\mu}{4}
      \right) \\
    &\ge S_{\ell_{\sigma_n}}^{(w)} + (Y_n-Y_{\sigma_n})
      + \frac{\eta_*\delta\mu}{4}\sum_{k=\sigma_n}^{n-1}
    \eta_{\ell_k+1}.
    \end{split}
\end{equation*}
As in the proof of Theorem
\ref{thm:one-d-bound}, this is sufficient to find a $\varepsilon>0$ such that 
\[
    \liminf_{n\to\infty} S_{n}^{(w)}  \ge \varepsilon.
\]
Finally, take a finite number of unit vectors 
$w_1,\ldots,w_N\in\mathcal{S}^d$ such that the corresponding segments
$\mathcal{S}(w_1),\ldots,\mathcal{S}(w_N)$
cover $\mathcal{S}^d$. Then,
\[
    \liminf_{n\to\infty} \inf_{v\in\mathcal{S}^d} v^T S_n v
    = \liminf_{n\to\infty}
    \min\big\{S_{n}^{(w_1)},\ldots,S_{n}^{(w_N)}\big\}  \ge \varepsilon.
    \qedhere
\]
\end{proof} 
\begin{lemma}
    \label{lemma:eigenbound} 
    Suppose $\F_n\subset\F_{n+1}$ are $\sigma$-algebras, and
    $G_{n+1}$ and $E_{n+1}$ are  $\F_{n+1}$-measurable 
    random variables, satisfying
    $\|E_{n+1}\|\le M$ for some constant $M<\infty$.
    Moreover, 
    $U_{n+2}$ is a random variable independent of $\F_{n+1}$,
    having a distribution $\nu$ fulfilling the conditions in Assumption
    \ref{a:uniform-component}.
    
    Let $\mathcal{S}^d\defeq \{u\in\R^d:\|u\|=1\}$ stand
    for the unit sphere and denote by
    $\mathcal{S}(w,\gamma) \defeq \{v\in\mathcal{S}^d: \|w-v\|\le
    \gamma\}$ the segment of the unit 
    sphere centred at $w\in\mathcal{S}^d$ and having the radius $\gamma>0$.
    There exist constants $\gamma,\mu>0$ such that
    \begin{equation*}
        \Pcond{\inf_{v\in\mathcal{S}(w,\gamma)}
          \big( 
            |v^T G_{n+1}|^2 
          + |v^T(U_{n+2} + \lambda G_{n+1} + E_{n+1})|^2
          \big) > \mu
          }{\F_n} \ge \frac{1}{2}.
    \end{equation*}
    for any $w\in\mathcal{S}^d$ and
    any constant $\lambda\in(0,1)$, 
    almost surely.
\end{lemma}
\begin{proof} 
Since $\nu$ is absolutely continuous with respect to the Lebesgue measure, 
one can show that there exist values $b,\gamma>0$ such that
\begin{equation}
  \inf_{w\in\mathcal{S}^d}\inf_{e\in B(0,M)} 
  \nu\left(\big\{u\in\R^d: \inf_{v\in
    \mathcal{S}(w,\gamma)} |v^T (u+e)|> b \big\}\right) \ge \frac{1}{2}
  \label{eq:cone-estim}
\end{equation}
where $B(0,M) \defeq \{y\in\R^d:\|y\|\le M\}$ denotes a 
centred ball of radius $M$.
Hereafter, fix $\gamma,b>0$ such that 
\eqref{eq:cone-estim} holds and let $a \defeq b/2$.

Fix a unit vector $w\in\mathcal{S}^d$ and
consider the set 
\[
    \begin{split}
    A &\defeq\left\{
      \inf_{v\in\mathcal{S}(w,\gamma)}
          \big( 
            |v^T G_{n+1}|^2 
          + |v^T(U_{n+2} + \lambda G_{n+1} + E_{n+1})|^2
          \big) \le a^2
       \right\} \\
     &\subset\left\{
      \inf_{v\in\mathcal{S}(w,\gamma)\;:\;|v^T G_{n+1}|\le a}
          |v^T(U_{n+2} + \lambda G_{n+1} + E_{n+1})|
       \le a \right\} \\
     &\subset\left\{
      \inf_{v\in\mathcal{S}(w,\gamma)\;:\;|v^T G_{n+1}|\le a}
          |v^T(U_{n+2} + E_{n+1})| - \lambda|v^T G_{n+1}|
       \le a \right\} \\
     &\subset\left\{
      \inf_{v\in\mathcal{S}(w,\gamma)}
          |v^T(U_{n+2} + E_{n+1})|
       \le 2a \right\}.
     \end{split}
\]
Since $U_{n+2}$ is independent of $\F_{n+1}$, and since 
$E_{n+1}$ is $\F_{n+1}$-measurable, one may 
estimate
\[
    \begin{split}
    \Pcond{A^\complement}{\F_{n}}
    &\ge \Econd{\inf_{e\in B(0,M)} 
    \Pcond{ \inf_{v\in\mathcal{S}(w,\gamma)}
      |v^T(U_{n+2}+e)| > 2a}{\F_{n+1}}}{\F_n}\\
    & = \inf_{e\in B(0,M)}  
    \nu\left(\big\{u\in\R^d: \inf_{v\in
    \mathcal{S}(w,\gamma)} |v^T (u+e)|> b \big\}\right)
    \ge \frac{1}{2}
    \end{split}
\]
by \eqref{eq:cone-estim}, almost surely, concluding the proof
by $\mu \defeq a^2$.
\end{proof}

\begin{corollary} 
    \label{cor:reg-cont} 
Assume $\pi$ is bounded, stays bounded away from zero on compact sets,
is differentiable on the tails, and has regular contours, that is,
\begin{equation}
    \liminf_{\|x\|\to\infty} \frac{x}{\|x\|} \cdot \frac{\nabla
      \pi(x)}{\|\nabla \pi(x)\|} < 0.
    \label{eq:reg-cont}
\end{equation}
Let $(X_n,M_n,S_n)_{n\ge 1}$ be an AM process as defined in Section
\ref{sec:notations} using a mixture proposal \eqref{eq:mix-proposal} with a mixing weight satisfying 
$\beta\in(0,1)$ and the density $\fixprop$
is bounded away from zero in some neighbourhood of the origin.
Moreover, suppose that the adaptation weights $(\eta_n)_{n\ge 2}$ satisfy 
Assumptions \ref{a:adapt-nonosc} and \ref{a:adapt-ell2-not-ell1}.
Then, \[
\liminf_{n\to\infty} \inf_{w\in\mathcal{S}^d} w^T S_n w >0.
\]
\end{corollary}
\begin{proof} 
  In light of Theorem \ref{th:am-zerobound}, it is sufficient to check
  Assumption \ref{a:uniform-component}, or in fact the conditions in Remark
  \ref{rem:uniform-continuous-component}. 
  Let $L>0$ be sufficiently large so that $\inf_{\|x\|\ge L} \frac{x}{\|x\|} \cdot \frac{\nabla
      \pi(x)}{\|\nabla \pi(x)\|} < 0$.
  \citeauthor{jarner-hansen}
  \cite[proof of Theorem 4.3]{jarner-hansen} show that
  there is an $\epsilon'>0$ and $K>0$ such that the cone 
  \begin{equation*}
      E(x)\defeq \left\{x-au:
        0<a<K,\,u\in \mathcal{S}^d,\,\left\|u-\frac{x}{\|x\|}\right\|\le \epsilon'
        \right\}  
  \end{equation*}
  is contained in the set $A(x)\defeq \{y\in\R^d:\pi(y)\ge \pi(x)\}$, for all $\|x\|\ge L$.
  
  Let $r'>0$ be sufficiently small to ensure that
  $\inf_{\|z\|\le r'} \fixprop(z) \ge \delta'>0$.
  There is a $r=r(\epsilon',K)\in(0,r'/2)$ and measurable
  $\xi:\R^d\to\R^d$ such that $\|\xi(x)-x\|\le r'/2$ and the ball 
  $B(x,r)\defeq \{y:\|y-\xi(x)\|\le r\}$ is contained in the cone $E(x)$.
  Define $\nu(x) \defeq c_r^{-1}\charfun{B(0,r)}(x)$ where $c_r\defeq |B(0,r)|$
  is the Lebesgue measure of $B(0,r)$,
  and let $\xi(x) \defeq x$ for the remaining $\|x\|<L$.
  Now, we have for $\|x\|\ge L$ that
  \[
      \beta \fixprop(x-y) \min\left\{1,\frac{\pi(y)}{\pi(x)}\right\} \ge
      \beta\delta'c_r \nu(y-\xi).
  \]
  Since $\pi$ is bounded and bounded away from zero on compact
  sets, the ratio $\pi(y)/\pi(x) \ge \delta''>0$ for all $x,y\in B(0,L+r')$ with
  $\|x-y\|\le r'$. Therefore, for all $\|x\|<L$, it holds that
  \[
      \beta \fixprop(x-y) \min\left\{1,\frac{\pi(y)}{\pi(x)}\right\} \ge
      \beta\delta'\delta''c_r\nu(y-x).\qedhere
  \]
\end{proof} 

\begin{remark} 
The conditions of Corollary \ref{cor:reg-cont} are fulfilled by many
practical densities $\pi$ (see \cite{jarner-hansen} for examples), and
are fairly easy to verify in practice.  Assumption \ref{a:uniform-component}
holds, however, more generally, excluding only densities with unbounded
density or having irregular contours.
\end{remark} 

\begin{remark} 
  It is not necessary for Theorem \ref{th:am-zerobound} and Corollary
  \ref{cor:reg-cont} to hold that 
  the adaptive proposal densities $\{\tilde{q}_s\}_{s\in\mathcal{C}^d}$
  have the specific form discussed in Section \ref{sec:notations}.
  The results require only that a suitable fixed proposal component is used
  so that Assumption \ref{a:uniform-component} holds.
  In Theorem \ref{th:ergodicity} below, however, the structure of
  $\{\tilde{q}_s\}_{s\in\mathcal{C}^d}$ is required.
\end{remark} 

Let us record the following ergodicity result, which is a counterpart to
\citep[Theorem 17]{saksman-vihola} formulating a
a strong law of large numbers for the original algorithm 
(S\ref{item:proposal})--(S\ref{item:adapt}) with the covariance parameter
\eqref{eq:orig-am-crec}.
\begin{theorem} 
    \label{th:ergodicity} 
    Suppose the target density $\pi$ is continuous and differentiable, stays bounded away
    from zero on compact sets and has super-exponentially decaying tails
    with regular contours,
    \[
        \limsup_{\|x\|\to\infty} \frac{x}{\|x\|^{\rho}} 
          \cdot \nabla \log \pi(x) = -\infty\qquad\text{and}\qquad
        \limsup_{\|x\|\to\infty} \frac{x}{\|x\|} 
          \cdot \frac{\nabla\pi(x)}{\|\nabla\pi(x)\|} < 0,
    \]
    respectively, for some $\rho>1$. 
    
   Let $(X_n,M_n,S_n)_{n\ge 1}$ be an AM process as defined in Section
   \ref{sec:notations} using a mixture proposal $q_s(z) =
   (1-\beta)\tilde{q}_s(z) + \beta \fixprop(z)$ where 
   $\tilde{q}_s$ stands for a zero-mean Gaussian
   density with covariance $s$, the mixing weight satisfies
   $\beta\in(0,1)$
   and the density $\fixprop$ is bounded away from zero in some
   neighbourhood of the origin. Moreover, suppose that the adaptation
   weights $(\eta_n)_{n\ge 2}$ satisfy Assumption
   \ref{a:adapt-ell2-not-ell1}. 
   
    Then, for any function $f:\R^d\to\R$ with
    $\sup_{x\in\R^d} \pi^{\gamma}(x) |f(x)| < \infty$ for some
    $\gamma\in(0,1/2)$, 
    \[
        \frac{1}{n} \sum_{k=1}^n f(X_k) \xrightarrow{n\to\infty} \int_{\R^d} f(x) \pi(x)
        \ud x
    \]
    almost surely.
\end{theorem} 
\begin{proof} 
The conditions of Corollary \ref{cor:reg-cont} are satisfied, implying that for 
any $\epsilon>0$ there is a $\kappa=\kappa(\epsilon)>0$ such that 
$\P\big(\inf_n \lambda_{\min}(S_n)  \ge \kappa\big) \ge 1-\epsilon$
where $\lambda_{\min}(s)$ denotes the smallest eigenvalue of $s$.
By \cite[Proposition 18]{saksman-vihola}, there is a compact set
$C_{\kappa}\subset\R^d$, a probability measure $\nu_\kappa$ on $C_\kappa$, 
and $b_\kappa<\infty$ such that for all $s\in\mathcal{C}^d$ with
$\lambda_{\min}(s)\ge \kappa$, it holds that
\begin{align}
    P_{\tilde{q}_s} V(x) &\le \lambda_s V(x)  + b \charfun{C_\kappa}(x), 
    & \forall x\in\R^d \label{eq:p-drift} \\
    P_{\tilde{q}_s}(x,A) &\ge \delta_s \nu(A) & \forall x\in 
    C_\kappa\label{eq:p-mino} 
\end{align}
where
$V(x)\defeq (\sup_x \pi(x))^{1/2} \pi^{-1/2}(x)\ge 1$ and the constants
$\lambda_s,\delta_s\in(0,1)$ satisfy the bound
\begin{equation}
   (1-\lambda_s)^{-1} \vee \delta_s^{-1} \le c_1 \det(s)^{1/2}
   \label{eq:const-estim}
\end{equation}
for some constant $c_1\ge 1$.
Likewise, there is a compact $D_f\subset\R^d$, a probability measure $\mu_f$
on $D_f$, and constants $b_f<\infty$ and
$\lambda_f,\delta_f\in(0,1)$, so that \eqref{eq:p-drift} and \eqref{eq:p-mino}
hold with $P_f$ \cite[Theorem 4.3]{jarner-hansen}.
Put together, \eqref{eq:p-drift} and \eqref{eq:p-mino} hold for $P_{q_s}$
for all $s\in\mathcal{C}^d$ with $\lambda_{\min}(s)\ge \kappa$, perhaps with different
constants, but
satisfying a bound \eqref{eq:const-estim}, with another $c_2\ge c_1$.

The rest of the proof follows as in Theorem \ref{thm:laplace-ergodicity}
by construction of an auxiliary process
$(\tilde{X}_n,\tilde{M}_n,\tilde{S}_n)_{n\ge 1}$ truncated so that 
for given $\varepsilon>0$,
$\kappa\le \lambda_{\min}(\tilde{S}_n)\le a n^\varepsilon$ and $|\tilde{M}_n| \le
a n^\varepsilon$ and where the constant $a=a(\varepsilon,\kappa)$
is chosen so that the truncated process 
coincides with the original AM process
with probability $\ge 1-2\epsilon$. Theorem 2 of \cite{saksman-vihola}
ensures that the strong law of large numbers holds for the constrained
process, and letting $\epsilon\to 0$ implies the claim.
\end{proof} 
\begin{remark} 
    In the case $\eta_n \defeq n^{-1}$, Theorem \ref{th:ergodicity}
    implies that with probability one, $M_n \to m_\pi \defeq \int x \pi(x) \ud x$ and 
    $S_n \to s_\pi \defeq \int xx^T \pi(x) \ud x - m_\pi m_\pi^T$, the
    true mean and covariance of $\pi$, respectively.
\end{remark}
\begin{remark} 
    Theorem \ref{th:ergodicity} holds also when using
    multivariate Student distributions $\{\tilde{q}_s\}_{s\in\mathcal{C}^d}$, as 
    \cite[Proposition
    26]{vihola-asm} extends \cite[Proposition 18]{saksman-vihola} to cover
    this case.
\end{remark} 


\section*{Acknowledgements} 

The author thanks Professor Eero Saksman for discussions and 
helpful comments on the manuscript.



\begin{thebibliography}{18}
\providecommand{\natexlab}[1]{#1}
\providecommand{\url}[1]{\texttt{#1}}
\expandafter\ifx\csname urlstyle\endcsname\relax
  \providecommand{\doi}[1]{doi: #1}\else
  \providecommand{\doi}{doi: \begingroup \urlstyle{rm}\Url}\fi

\bibitem[Andrieu and Moulines(2006)]{andrieu-moulines}
C.~Andrieu and {\'E}.~Moulines.
\newblock On the ergodicity properties of some adaptive {MCMC} algorithms.
\newblock \emph{Ann. Appl. Probab.}, 16\penalty0 (3):\penalty0 1462--1505,
  2006.

\bibitem[Andrieu and Robert(2001)]{andrieu-robert}
C.~Andrieu and C.~P. Robert.
\newblock Controlled {MCMC} for optimal sampling.
\newblock Technical Report Ceremade 0125, Universit{é} Paris Dauphine, 2001.

\bibitem[Andrieu and Thoms(2008)]{andrieu-thoms}
C.~Andrieu and J.~Thoms.
\newblock A tutorial on adaptive {MCMC}.
\newblock \emph{Statist. Comput.}, 18\penalty0 (4):\penalty0 343--373, Dec.
  2008.

\bibitem[Atchad{\'e} and Fort(2009)]{atchade-fort}
Y.~Atchad{\'e} and G.~Fort.
\newblock Limit theorems for some adaptive {MCMC} algorithms with subgeometric
  kernels.
\newblock \emph{Bernoulli}, 2009.
\newblock to appear.

\bibitem[Atchad{\'e} and Rosenthal(2005)]{atchade-rosenthal}
Y.~F. Atchad{\'e} and J.~S. Rosenthal.
\newblock On adaptive {M}arkov chain {M}onte {C}arlo algorithms.
\newblock \emph{Bernoulli}, 11\penalty0 (5):\penalty0 815--828, 2005.

\bibitem[Athreya and Ney(1978)]{athreya-ney}
K.~B. Athreya and P.~Ney.
\newblock A new approach to the limit theory of recurrent {M}arkov chains.
\newblock \emph{Trans. Amer. Math. Soc.}, 245:\penalty0 493--501, 1978.

\bibitem[Bai et~al.(2008)Bai, Roberts, and Rosenthal]{bai-roberts-rosenthal}
Y.~Bai, G.~O. Roberts, and J.~S. Rosenthal.
\newblock On the containment condition for adaptive {M}arkov chain {M}onte
  {C}arlo algorithms.
\newblock Preprint, July 2008.
\newblock URL \url{http://probability.ca/jeff/research.html}.

\bibitem[Esseen(1966)]{esseen-kolmogorov-rogozin}
C.~G. Esseen.
\newblock On the {K}olmogorov-{R}ogozin inequality for the concentration
  function.
\newblock \emph{Z. Wahrscheinlichkeitstheorie verw. Gebiete}, 5\penalty0
  (3):\penalty0 210--216, Sept. 1966.

\bibitem[Haario et~al.(2001)Haario, Saksman, and Tamminen]{saksman-am}
H.~Haario, E.~Saksman, and J.~Tamminen.
\newblock An adaptive {M}etropolis algorithm.
\newblock \emph{Bernoulli}, 7\penalty0 (2):\penalty0 223--242, 2001.

\bibitem[Hall and Heyde(1980)]{hall-heyde}
P.~Hall and C.~C. Heyde.
\newblock \emph{Martingale Limit Theory and Its Application}.
\newblock Academic Press, New York, 1980.
\newblock ISBN 0-12-319350-8.

\bibitem[Jarner and Hansen(2000)]{jarner-hansen}
S.~F. Jarner and E.~Hansen.
\newblock Geometric ergodicity of {M}etropolis algorithms.
\newblock \emph{Stochastic Process. Appl.}, 85:\penalty0 341--361, 2000.

\bibitem[Nummelin(1978)]{nummelin-splitting}
E.~Nummelin.
\newblock A splitting technique for {H}arris recurrent {M}arkov chains.
\newblock \emph{Z. Wahrscheinlichkeitstheorie verw. Gebiete}, 43\penalty0
  (3):\penalty0 309--318, Dec. 1978.

\bibitem[Roberts and Rosenthal(2007)]{roberts-rosenthal}
G.~O. Roberts and J.~S. Rosenthal.
\newblock Coupling and ergodicity of adaptive {M}arkov chain {M}onte {C}arlo
  algorithms.
\newblock \emph{J. Appl. Probab.}, 44\penalty0 (2):\penalty0 458--475, 2007.

\bibitem[Roberts and Rosenthal(2009)]{roberts-rosenthal-examples}
G.~O. Roberts and J.~S. Rosenthal.
\newblock Examples of adaptive {MCMC}.
\newblock \emph{J. Comput. Graph. Statist.}, 18\penalty0 (2):\penalty0
  349--367, 2009.

\bibitem[Roberts and Rosenthal(2004)]{roberts-rosenthal-general}
G.~O. Roberts and J.~S. Rosenthal.
\newblock General state space {M}arkov chains and {MCMC} algorithms.
\newblock \emph{Probability Surveys}, 1:\penalty0 20--71, 2004.

\bibitem[Rogozin(1961)]{rogozin-concentration}
B.~A. Rogozin.
\newblock An estimate for concentration functions.
\newblock \emph{Theory Probab.~Appl.}, 6\penalty0 (1):\penalty0 94--97, Jan.
  1961.

\bibitem[Saksman and Vihola(2009)]{saksman-vihola}
E.~Saksman and M.~Vihola.
\newblock On the ergodicity of the adaptive {M}etropolis algorithm on unbounded
  domains.
\newblock Preprint, arXiv:0806.2933v2, Feb. 2009.

\bibitem[Vihola(2009)]{vihola-asm}
M.~Vihola.
\newblock On the stability and ergodicity of an adaptive scaling {M}etropolis
  algorithm.
\newblock Preprint, arXiv:0903.4061v1, Mar. 2009.

\end{thebibliography}

\appendix

\section{The Kolmogorov-Rogozin Inequality}
\label{sec:kolmogorov-rogozin} 

Define the concentration function $Q(X;\lambda)$ of a random variable $X$
by
\[
    Q(X;\lambda) \defeq \sup_{x\in\R} \P(X\in [x, x + \lambda])
\]
for all $\lambda\ge 0$.
\begin{theorem} 
    \label{th:kolmogorov-rogozin}
    Let $X_1,X_2,\ldots$ be mutually independent random variables.
    There is a universal constant $c>0$ such that
    \[
        Q\left(\sum_{k=1}^n X_k; L\right)
        \le \frac{c L}{\lambda} \left(
          \sum_{k=1}^n \big(1-Q(X_k;\lambda)\big)
        \right)^{-1/2}
    \]
    for all $L\ge \lambda > 0$.
\end{theorem}
\begin{proof} 
    Rogozin's original work \cite{rogozin-concentration} uses
    combinatorial results, and Esseen's alternative proof
    \cite{esseen-kolmogorov-rogozin} is based on characteristic functions.
\end{proof} 


\section{A Coupling Construction}
\label{sec:coupling} 

\begin{theorem}
    \label{thm:coupling} 
    Suppose $\mu$ and $\nu$ are probability measures and the random
    variable $X\sim \mu$.
    Then, possibly by augmenting the probability space, 
    there is another random variable $Y$ such that $Y\sim \nu$
    and $\P(X = Y) = 1 - \|\mu - \nu\|$.
\end{theorem}
\begin{proof}[Proof (adopted from Theorem 3 in
    \citep{roberts-rosenthal-general}).] 
    Define the measure $\rho \defeq \mu + \nu$, and the densities $g
    \defeq \ud \mu / \ud \rho$ and $h \defeq \ud \nu/\ud \rho$,
    existing by the Radon-Nikodym theorem.
    Let us introduce
    two auxiliary variables $U$ and $Z$
    independent of each other and $X$, whose existence is ensured by
    possible augmentation of the probability space.
    Then, $Y$ is defined through
    \[
        Y = \charfun{\{U\le r(X)\}}
        X + \charfun{\{U> r(X)\}} Z
    \]
    where the `coupling probability' $r$ is defined as 
    $r(y) \defeq \min\{1,h(y)/g(y)\}$ whenever
    $g(y)>0$ and $r(y)\defeq 1$ otherwise.
    The variable $U$ is uniformly distributed on $[0,1]$.
    If $r(y) = 1$ for
    $\rho$-almost every $y$, then the choice of $Z$ is irrelevant,
    $\mu = \nu$, and the claim is trivial.
    Otherwise,
    the variable $Z$ is distributed following
    the `residual measure'
    $\xi$ given as
    \[
        \xi(A) \defeq
        \frac{\int_A
          \max\{0,h-g\} \ud \rho
          }{\int
          \max\{0,h-g\} \ud \rho}.
    \]
    Observe that $\int
    \max\{0,h-g\} \ud \rho =\int
    \max\{0,g-h\} \ud \rho > 0$ in this case, so $\xi$ is a
    well defined probability measure.
    
    Let us check that $Y\sim \nu$,
    \[
        \begin{split}
        \P(Y\in A)
        &= \int_A r \ud \mu
        + \xi(A) \int (1-r) \ud \mu\\
        &= \int_A \min\{g,h\} \ud\rho 
        + \xi(A) \int_{h<g} (g-h) \ud\rho \\
        &= \int_A \min\{g,h\} + \max\{0,h-g\}\rho(\ud x) 
        = \nu(A).
        \end{split}
    \]
    Moreover, by observing that $r(y)=1$ in the support of $\xi$, one has
    \begin{equation*}
        \P(X = Y) = \int r \ud\mu = \int \min\{g,h\}
        \ud\rho 
        = 1 - \int_{g<h} (h - g) \ud\rho
        = 1 - \| \nu - \mu \|
    \end{equation*}
    since $\int_{g<h} (h - g) \ud\rho = \int_{h<g} (g - h)
    \ud\rho = \sup_{f} \left|\int f (h - g) \ud\rho\right|
    = \| \mu - \nu \|$
    where the supremum taken over all measurable functions $f$
    taking values in $[0,1]$.
\end{proof}


\section{Proof of Lemma \ref{lemma:exp-geombound}}
\label{sec:proof-exp-geombound} 
Observe that without loss of generality it is sufficient to check 
the case $m=0$ and $b=1$, that is, 
consider the standard Laplace distribution $\pi(x) \defeq \frac{1}{2}e^{-|x|}$.

Let $x>0$ and start by writing
\begin{equation}
1 - \frac{P_s V(x)}{V(x)}
= 
    \int_{-x}^x a(x,y)
    \tilde{q}_s(y-x) \ud y
  - \int_{|y|>x} b(x,y)
    \tilde{q}_s(y-x) \ud y
    \label{eq:drift-estim}
\end{equation}
where 
\begin{eqnarray*}
    a(x,y) &\defeq& \left( 1 - \sqrt{\frac{\pi(x)}{\pi(y)}}\right)
    = 1 - e^{-\frac{x-|y|}{2}} \qquad\text{and}\\
    b(x,y) &\defeq& \sqrt{\frac{\pi(y)}{\pi(x)}} 
    \left(1-\sqrt{\frac{\pi(y)}{\pi(x)}}\right)
    = e^{-\frac{|y|-x}{2}}
    \left(1-e^{-\frac{|y|-x}{2}}\right).
\end{eqnarray*}
Compute then that
\[
    \int_0^x a(x,y) \tilde{q}_s(y-x) \ud y - \int_x^{2x} b(x,y) \tilde{q}_s(y-x)\ud y
    = \int_0^x \big( 1 - e^{-\frac{z}{2}}\big)^2 \tilde{q}_s(z)\ud z .
\]
The estimates
\begin{equation*}
    \begin{split}
    \int_{-x}^0 a(x,y) \tilde{q}_s(y-x) \ud y
    &\ge \tilde{q}_s(2x)\int_0^x a(x,y) \ud y 
    = \tilde{q}_s(2x)\int_0^x (1-e^{-\frac{z}{2}}) \ud z \\
    \int_{-\infty}^{-x} b(x,y) \tilde{q}_s(y-x) \ud y
    & \le \tilde{q}_s(2x) \int_{x}^{\infty} b(x,y) \ud y
    = \tilde{q}_s(2x) \int_{0}^{\infty} e^{-\frac{z}{2}}(1-e^{\frac{z}{2}})
    \ud z
    \end{split}
\end{equation*}
due to the non-increasing property of $\tilde{q}_s$ yield
\begin{multline*}
    \int_{-x}^0 a(x,y) \tilde{q}_s(y-x) \ud y
    -\int_{-\infty}^{-x} b(x,y) \tilde{q}_s(y-x) \ud y\\
    \ge \tilde{q}_s(2x)
    \left[
        \int_0^x (1-e^{-\frac{z}{2}})^2 \ud z
        - \int_x^\infty e^{-\frac{z}{2}} \ud z
    \right]>0
\end{multline*}
for any sufficiently large $x> 0$. Similarly, one obtains
\[
    \frac{1}{2}\int_0^x \big( 1 - e^{-\frac{z}{2}}\big)^2 \tilde{q}_s(z)\ud z
    - \int_{2x}^\infty b(x,y) q_s(y-x) \ud y > 0
\]
for large enough $x>0$.

Summing up, letting $M>0$ be sufficiently large, then for 
$x\ge M$ and $s\ge L>0$
\begin{equation*}
    \begin{split}
1 - \frac{P_s V(x)}{V(x)}
&\ge \frac{1}{2}\int_0^x \big( 1 - e^{-\frac{z}{2}}\big)^2 \tilde{q}_s(z)\ud z
\ge \frac{1}{2}\tilde{q}_s(M)\int_0^M \big( 1 - e^{-\frac{z}{2}}\big)^2 \ud
z \\
&\ge c_1 s^{-1/2} \tilde{q}(\theta^{-1/2} s^{-1/2} M) \ge c_2 s^{-1/2}
\end{split}
\end{equation*}
for some constants $c_1,c_2>0$.
The same inequality holds also for $-x\le -M$ due to symmetry.
The simple bound $P_s V(x) \le 2 V(x)$ observed from 
\eqref{eq:drift-estim} with the above estimate establishes
\eqref{eq:drift-ineq}. The minorisation inequality
\eqref{eq:mino-ineq} holds since for all $x\in C$ one may write
\begin{equation*}
    \begin{split}
    P_s(x,\set{A}) 
    &\ge \int_{A\cap C} 
    \max\left\{1,\frac{\pi(y)}{\pi(x)}\right\} \tilde{q}_s(y-x) \ud y  \\
    &\ge 
    \frac{\inf_{z\in\set{C}}\pi(z)}{\sup_z\pi(z)}
    \inf_{s\ge L,\, z,y\in C} \tilde{q}_s(z-y)
    \int_{A\cap C}
    \ud y \ge c_3 s^{-1/2} \nu(A).
\end{split}
\end{equation*}
where $\nu(A) \defeq |A\cap C|/|C|$ with $|\cdot|$ denoting the Lebesgue
measure. \qed

\end{document}